\documentclass[11pt]{article}
\usepackage{verbatim} 
\usepackage{graphicx}
\usepackage[utf8]{inputenc}
\usepackage{amsmath,amsfonts,amssymb,amscd,amsthm,txfonts,hyperref}
\usepackage{xcolor}
\newtheorem{theorem}{Theorem}[section]
\newtheorem{proposition}[theorem]{Proposition}
\newtheorem{lemma}[theorem]{Lemma}
\newtheorem{corollary}[theorem]{Corollary}
\newtheorem{remark}{Remark}[section]

\theoremstyle{definition}
\newtheorem{definition}[theorem]{Definition}

\newcommand{\R}{\mathbb{R}}

\newcommand{\N}{\mathbb{N}}
\newcommand{\Z}{\mathbb{Z}}
\newcommand{\T}{\mathbb{T}}

\newcommand{\lap}{\bigtriangleup}

\newcommand{\E}{\mathbb E}

\newcommand{\re}{\textrm{Re}}
\newcommand{\ent}[1]{\lfloor #1\rfloor}
\newcommand{\mE}{\mathcal E}
\newcommand{\mH}{\mathcal H}

\begin{document}
\newcommand{\Hloc}{H^{1/2-}_{\textrm{loc}}}

\title{Invariant measure for the Klein-Gordon equation in a non periodic setting}
\author{Anne-Sophie de Suzzoni\footnote{Universit\'e Paris 13, Sorbonne Paris Cit\'e, LAGA, CNRS ( UMR 7539), 99, avenue Jean-Baptiste Cl\'ement, F-93430 Villetaneuse, France}}

\maketitle

\begin{abstract} In this paper, we build a Gibbs measure for the 1d cubic Klein-Gordon equation on $\R$ with a decreasing non linearity, in the sense that the non linearity $f^3$ is multiplied by $\chi$ where $\chi$ is a sufficiently integrable non negative function. We prove that this equation is almost surely globally well-posed in $\Hloc$ with respect to this measure and that the measure is invariant under the flow of the equation. \end{abstract}

\tableofcontents

\section{Introduction}

Our aim in this paper is to build a Gibbs measure such as in \cite{TVinvBO, NORSinv, BTTlong,dSinv,Ohinv,Binvkdv} for the 1 dimensional defocusing cubic Klein-Gordon equation with a decreasing non linearity : 
\begin{equation}\label{nlkg}
\partial_t^2 f + (1-\partial_x^2) f + \chi f^3 = 0
\end{equation}
where $\chi$ is a sufficiently integrable non negative function, and where the variable $x$ belongs to $\R$. We prove the invariance of this Gibbs measure under the flow of \eqref{nlkg}. 

As inspired by the works of Lebowitz-Rose-Speers \cite{LRS}, Bourgain \cite{Binvkdv}, Zhidkov \cite{Zoni}, Burq-Thomann-Tzvetkov \cite{BTTlong}, a Gibbs measure is usually built in the following way. We consider a Hamiltonian equation with Hamiltonian
$$
H(u) = H_c(u) + H_p(u) \; .
$$
where $H_p$, the potential energy, corresponds to the non linearity of the equation and $H_c$, the kinetic energy, can be written 
$$
H_c(u) = \langle u, Lu \rangle_X
$$
where $L$ is a Hermitian positive operator on a Hilbert space $X$ ($\langle, \rangle_X$ is the scalar product on $X$). In the case of \eqref{nlkg}, this operator is $1-\lap$. 

In a first time, let us admit that $X$ and $L$ are such that there exists an orthonormal basis of $X$, $(e_n)_{n\in \N}$ satisfying for all $n$, $Le_n = \lambda_n e_n$, with $\lambda_n > 0$, as is the case in \cite{TVinvBO,BTTlong,Binvkdv,dSinv}. Let $(g_n)_{n\in \N}$ be a sequence of independent Gaussian variables of law $\mathcal N (0,1)$ on a probability space $(\Omega, \mathcal A, \mathbb P)$. We define a sequence of random variables
$$
\varphi_N  = \sum_{n=0}^N \frac1{\sqrt{\lambda_n}} e_n g_n\; .
$$
The sequence $(\varphi_N)_N$ converges in a Hilbert space $L^2(\Omega,Y)$, for instance with
$$
Y = \big\{ u = \sum_n u_n e_n \; \Big| \; \sum_n \lambda_n n^{-(1+\varepsilon)} |u_n|^2 < \infty \big\}
$$
for $\varepsilon > 0$, towards a random variable $\varphi$. By calling $\mu$ the measure on $Y$ defined as 
$$
\mu(A) = \mathbb P (\varphi^{-1}( A))
$$
for all measurable sets $A$, we can interpret $\mu$ as 
\begin{eqnarray*}
d\mu \Big( \sum_n u_n e_n \Big) & " = " & e^{- \sum_n \lambda_n^2 |u_n|^2} \prod_n \frac{\lambda_n du_n d\overline{u_n}}{2\pi} \\
& "="& de^{-H_c(u)}dL(u) 
\end{eqnarray*}
where $L$ would be the Lebesgue measure and $d$ a normalisation factor. This identity makes sense in finite dimension. 

The Gibbs measure $\rho$ is built as 
$$
d\rho(u) = \frac{1}{D}e^{-H_p(u)}d\mu(u) 
$$
where $D$ is a normalisation factor.

We can interpret $\rho$ as 
$$
d\rho (u) "=" \frac{d}{D}e^{-H(u)}dL(u)\; .
$$
In \cite{TVinvBO,BTTlong,Binvkdv,dSinv}, the proofs of the invariance of the Gibbs measure associated to an equation under the flow of this equation use the fundamental arguments that the Lebesgue measure is invariant under Hamiltonian flows, and that the Hamiltonian is a conserved quantity.

For the Klein-Gordon equation in a periodic setting, $L= 1-\lap$ and $X = L^2(\T)$ satisfy these properties, as $Le_n = (1+n^2) e_n$ with $e_n : x\mapsto e^{inx}$. The method presented above produces a measure on $H^s(\T)$ with $s<1/2$. This measure is invariant under the flow of the periodic cubic Klein-Gordon equation.

Here, we are interested in the non-periodic setting. Our aim is to build a measure of the form
$$
d\rho(u) "=" \frac{d}{D}e^{-H(u)} dL(u) \; .
$$
We call $(W_n)_{n\in \R}$ a Brownian motion and 
$$
\varphi_{N,M} (\omega, x) = \sum_{n= -NM}^{NM-1} \Big( W_{\frac{n+1}{N}} - W_{\frac{n}{N}} \Big) \frac{1}{\sqrt{1+ \frac{n^2}{N^2}}} e^{ixn/N} \; .
$$
By comparison with the periodic setting, $1+ \frac{n^2}{N^2} $ plays the role of $\lambda_n$ and $W_{\frac{n+1}{N}} - W_{\frac{n}{N}}$ the one of $g_n$. At $M$ and $x$ fixed, $(\varphi_{N,M}(.,x))_{N\in \N^*}$ converges towards a Itô integral
$$
\varphi_M (.,x)  = \int_{-M}^M \frac1{\sqrt{1+n^2}} e^{inx} dW_n\; .
$$
The limit $M \rightarrow \infty$ is similar to its counterpart in the periodic setting. We get a limit $\varphi$ in
$$
\Hloc  = \big\{ u \; \Big| \; \forall R , s<1/2, \; p_{R,s} (u) = \int_{-R}^R | (1-\lap)^{s/2} u|^2(x) dx < \infty \big\}
$$ 
whose topology is the metric induced by the distance
$$
d(u,v)  = \sum_{R\in \N^*} \sum_{l\in \N^*} 2^{-(R+l)} \frac{p_{R,1/2-1/l}(u-v)}{1+ p_{R,1/2-1/l}(u-v)} \; .
$$
This space is complete. This measure appeared previously in the xork by McKean-Vaninsky, \cite{MVsta}.

We could have chosen another topology for which $\varphi_{N,M}$ would have converged but this one has the advantage that the flow of the linear Klein-Gordon equation, $e^{it\sqrt{1-\lap}}$ is continuous for the distance $d$, a property needed would it be only to give a meaning to the definition of the invariance of $\rho$.

The measure $\mu$ induced by $\varphi$ is invariant under the linear flow. 

We built the Gibbs measure $\rho$ for \eqref{nlkg} as 
$$
d\rho(u) = \frac1{D}e^{-\frac12 \int \chi |u|^4} d\mu(u)\; ,
$$
in an analogous way to the periodic case. We prove that if the initial datum of \eqref{nlkg} is taken in the support of $\rho$, then there exists a unique global solution. We call the flow of \eqref{nlkg} $\psi(t)$. We prove that $\psi(t)$ is continuous on the support of $\rho$, hence measurable, in $\Hloc$ with regard to the initial datum, and finally that $\rho$ is invariant under the flow $\psi(t)$, that is, for all measurable set $A \subset \Hloc$,
$$
\rho (\psi(t)^{-1}(A)) = \rho(A) \; .
$$

We remark that the restriction on $\chi$, detailed in Section 2, are probably not optimal. We also note that for all $u\in \Hloc$, $u$ is $\rho$-almost surely not in $L^2$. It is unclear to us whether the equation \eqref{nlkg} is well-posed or ill-posed when the initial datum is in in $\Hloc$. The local well-posedness is made easier by the presence of $\chi$ in the non-linearity but the global well-posedness cannot be obtained using the same energy estimates as in Subsection \ref{subsec-exten}. We mention \cite{CCTill} for proofs of ill-posedness in the close case of the defocusing non linear wave equation. For the local well-posedness, we use the same techniques as in \cite{twon}.

The choice of Klein-Gordon has been made because the propagation speed is finite (the linear flow is continuous for the topology of $\Hloc$, as opposed to the Schrödinger equation) and because the small frequencies are not an obstacle to the construction of the measure (as opposed to the wave equation).

Let us state the result more precisely.

\begin{theorem} There exists a measure $\rho$ on $\Hloc$ such that
\begin{itemize}
\item for all $u$ in the support of $\rho$, $u$ does not belong to $L^2$,
\item the flow $\psi(t)$ of \eqref{nlkg} is globally well-defined on the support of $\rho$ and is continuous with respect to the initial datum for the topology of $\Hloc$,
\item the measure $\rho$ is invariant under the flow $\psi(t)$.
\end{itemize}
\end{theorem}

We remark that this leads to the existence of at least one solution with initial datum not in $L^2$.

The strategy to prove the invariance and the global well-posedness of the equation is similar to the one in the periodic case. As in previous works , we approach $\psi(t)$ by a sequence $(\psi_k(t))_k$, where $\psi_k(t)$ is the flow of a Hamiltonian equation on a finite dimensional space $E_k$. We then approach $\rho$ by measures $\rho_k$ whose supports are included in $E_k$. 

We prove that $\rho_k$ converges weakly towards $\rho$ and that $\rho_k$ is invariant under $\psi_k(t)$ with classical methods such as Liouville's theorem. We extend $\psi_k(t)$ to $\Hloc$. We then prove that locally in time, \eqref{nlkg} is well-posed and that $\psi_k(t)u$ converges towards $\psi(t) u$ in $\Hloc$. We prove that $\psi_k(t)$ and $\psi(t)$ are continuous with respect to the initial datum globally and locally in time respectively. These properties on $\psi$ and $\psi_k$ do not require any probabilistic arguments except knowing where the support of $\rho$ is included.

We build a set $A$ of full $\rho$-measure such that for all $u \in A$, $\psi_k(t) u$ is controlled for a set of countable times $t_n$. This is possible because $\psi_k(t_n)$ is continuous on $\Hloc$, $\rho_k$ converges weakly towards $\rho$, and $\rho_k$ is invariant under $\psi_k(t)$. We then propagate the local properties (existence and uniqueness of $\psi(t)$, continuity of $\psi(t)$ with respect to the initial datum, convergence of $\psi_k(t)u$ towards $\psi(t)u$) from $[t_{n-1},t_{n}]$ to $[t_n,t_{n+1}]$ to make them global. This is possible thanks to the finite times controls. 

Finally, we prove the invariance of $\rho$ under $\psi(t)$, following the heuristic idea that for a measurable set $B$, $\psi(t)^{-1}(B)$ is the limit of $\psi_k(t)^{-1}(B)$, $\rho$ is the limit of $\rho_k$ and $\rho_k(\psi_k(t)^{-1}(B)) = \rho_k(B)$.

The main difficulty compared to the periodic case, apart from defining the measure, is the particular attention that has to be paid to the continuity of the flows with respect to the initial datum, because of our less usual topology. 

We also mention the works by Bourgain \cite{Binvinf} and Xu \cite{samxu}, which are close to our study.

\paragraph{Organisation of the paper} In Section 2, we build a measure $\mu$ on $\Hloc$ following the idea mentioned above and we prove its invariance under the flow of the linear Klein-Gordon equation.

In Section 3, we build the Gibbs measure $\rho$ and its approximation $\rho_k$ and prove the weak convergence of $\rho_k$ towards $\rho$. We also introduce the approaching flows $\psi_k$ on the finite dimensional spaces and prove that $\rho_k$ is invariant under $\psi_k(t)$. 

In Section 4, we prove the local well-posedness of \eqref{nlkg}, extend $\psi_k(t)$ to $\Hloc$, prove that locally in time $\psi_k(t)u$ converges towards $\psi(t) u$ in $\Hloc$, and that $\psi_k(t)$ and $\psi(t)$ are continuous with respect to the initial datum. 

In Section 5, we extend the local properties to global times and prove the invariance of $\rho$ under $\psi(t)$.

Notation : In the rest of the paper, the norm $\|.\|_{X_*,Y_\circ}$ means
$$
\|u\|_{X_*,Y_\circ} = \| * \mapsto \|\circ \mapsto u(x,\circ) \|_Y \|_X
$$
where $*$ and $\circ$ may be replaced by time variables ($t$, $\tau$), space variables ($x$), or random events ($\omega$). The space $X_*,Y_\circ$ is the space of functions normed by $\|.\|_{X_*,Y_\circ}$.

\section{Invariance under the linear flow}
We study the equation 
\begin{equation}\label{maineq} 
\left \lbrace{\begin{tabular}{ll}
$\partial_t^2 f + (1-\lap) f + \chi f^3 = 0$ \\
$f_{|t=0} = f_0 \; ,\; (\partial_t f)_{|t=0} = f_1$
\end{tabular}} \right.
\end{equation}
We assume that $\chi$ satisfies 
$$
0 \leq \chi(x) \leq C (\sqrt{1+x^2})^{-3\alpha}
$$
for some $\alpha > 1$. In particular, $\chi$ belongs to $L^1 \cap L^\infty$ and  $\sqrt{1+x^2}\chi$ belongs to $L^1$.

For instance, if $A$ is a measurable set of $\R$ with finite Lebesgue measure, we can take $\chi = \frac1{\sqrt{1+x^2}} 1_A$ where $ 1$ is the indicative function. Note that we do not require more regularity on $\chi$. 

The fact that $\chi \geq 0$ is due to energetic reasons. With this restriction, we have that the potential energy is of the same sign as the kinetic energy. In other words, the equation is defocusing.

The restriction $\sqrt{1+x^2} \chi \in L^1$ appears in the definition of the Gibbs measure and the proofs of some of its properties.

The restriction $ \chi \in L^\infty$ appears in the analysis of the PDE, it is technical and could probably be improved by a more careful analysis.

If $K$ is a measurable bounded set of $\R$, then we can take $\chi = 1_K$. We can interpret this $\chi$ as the fact that the system described by $f$ evolves in a medium that allows interactions in $K$ but not in its complementary in $\R$. Note that $K$ does not have to be compact. For other $\chi$, we can interpret that the interactions are depending on where they occur.

For the rest of the paper, we replace the equation on $f$ by the equivalent equation on $u = f + i(1-\lap)^{-1/2} \partial_t f$. We have that $f = \re u$ solves \eqref{maineq} if and only if $u$ solves 
\begin{equation}\label{NLKG}
\left \lbrace{ \begin{tabular}{ll}
$i\partial_t u = \sqrt{1-\lap} u + (1-\lap)^{-1/2} (\chi (\re u)^3) $ \\
$ u_{|t=0} = u_0 := f_0 + i (1-\lap)^{-1/2} f_1 $ \end{tabular}}\right. \; .
\end{equation}

The linear part of \eqref{NLKG} is 
$$
\left \lbrace{\begin{tabular}{ll}
$i\partial_t u = \sqrt{1-\lap} u$ \\
$u_{|t=0} = u_0$ \end{tabular}} \right. \; .
$$
In this section, we focus on this equation. We can write its solution $u(t) = L(t) u_0$ where $L(t)$ is the Fourier multiplier
$$
\widehat{L(t) f}(n) = e^{-it\sqrt{1+n^2}} \hat f (n) \; .
$$
We build here a measure $\mu$ invariant under the flow $L(t)$. However, although we need the measure $\mu$ to define $\rho$, the measure invariant under the non linear flow, the proof of the invariance of $\mu$ is not required to prove the invariance of $\rho$. This is one of the strategical differences from the compact setting where it is more convenient to use the invariance of $\mu$. Here, this strategy does not apply. Nevertheless, we prove the invariance of $\mu $ for the following reason. The proof of this invariance present the same difficulties in probability as the proof of the invariance of $\rho$, but without the difficulties due to the analysis of the studied PDE.

\subsection{Definition of the random variable \texorpdfstring{$\varphi$}{phi}}

We define $\mu$, the invariant measure under the linear flow, as the image measure through a random variable $\varphi$. In this subsection, we define $\varphi$ and give some of its properties. 

\paragraph{Definition of \texorpdfstring{$\varphi$}{phi}}

In this paragraph, we define the Gaussian variable $\varphi$ which we use to define the measure $\mu$ invariant under the linear flow.

For the rest of this paper, we call $(\Omega, \mathcal F, \mathbb P)$ a probability space and $(W_n)_{n\in \R}$ a complex Brownian motion on this space, or more accurately, the union of two Brownian motions with the same initial value.

The random variable $\varphi$ is defined as the limit of a sequence of random variables. Let us describe this sequence.

\begin{definition}\label{def-phi} Let $N,R\in \N$. We call $\varphi_{N,R}$ the random variable defined as : 
$$
\varphi_{N,R} (\omega, x) = \sum_{k=-NR}^{NR-1} \delta_{N,k} (\omega) \frac1{\sqrt{1+\frac{k^2}{N^2}}}e^{ikx/N}
$$
where $\omega \in \Omega$ is an event of the probability space, $x\in \R$ is the space variable and $\delta_{N,k} = W_{\frac{k+1}{N}}-W_{\frac{k}{N}}$.
\end{definition}

\begin{remark}\label{rem-Glaw} This random variable is a Gaussian vector. Indeed, $\varphi_{N,R}$ is entirely determined by $2N$ Gaussian variables $a_{-NR},\hdots, a_{NR-1}$ with $a_k= \delta_{k,N} (1+\frac{k^2}{n^2})^{-1/2}$. The law of $a_k$ is $\mathcal N(0, \frac1{N(1+ k^2/N^2)})$ and the $a_k$ are independent from each other. Hence they form a Gaussian vector whose law is given by the covariance matrix $M(N)$ such that
$$
M(N)_{i,j} = \E (\overline a_i a_j) = \delta_i^j \frac1{N(1+\frac{j^2}{N^2})}\; ,
$$
with $\delta_i^j = 1$ f $i=j$ and $0$ otherwise.

Its law is given by 
$$
(\mbox{det } M(N))^{-1/2}e^{-\langle a,M(N)^{-1} a\rangle}\prod_{k= -NR}^{NR-1} \frac{da_kd\overline a_k}{2\pi} 
$$
where
$$
\langle a,M(N)^{-1}a \rangle = \sum_{k= -NR}^{NR-1} |a_k|^2 N \left( 1+\frac{k^2}{N^2}\right)
$$
can be rewritten as 
$$
\frac1{2\pi}\int_{-\pi N}^{\pi N} \overline v(x) (1-\Delta)v(x)dx
$$
where $v$ is given by
$$
v(x) = \sum_{k=-RN}^{RN-1} a_k e^{ikx/N}\; .
$$
\end{remark}

\begin{lemma}Let $D = \sqrt{1-\lap}$. Let $s<1/2$. The sequence $D^s\varphi_{2^n, R}$ converges in \\$\sqrt{1+|x|^2} L^\infty_\R, L^2_\Omega$ when $n$ goes to $\infty$, uniformly in $R$. We call its limit $D^s\varphi_R$. In other words, for all $\varepsilon >0$, there exists $n_0\in \N$ such that for all $R$ and all $n\geq n_0$, 
$$
\|(1+x^2)^{-1/2}(D^s \varphi_R - D^s\varphi_{2^n,R})\|_{L^\infty_\R,L^2_\Omega}\leq \varepsilon \; .
$$\end{lemma}

We remark that for all $\omega \in \Omega$, $\varphi_{N,R} $ belongs to the dual of the Schwartz functions, hence $D^s$ is defined as the Fourier multiplier $(1+n^2)^{s/2}$. It is consistent with seeing $D^s$ as an operator on $L^2(\R/ (2\pi N \Z))$.

\begin{proof}We begin by writing $D^s\varphi_{N,R}$. We have : 
$$
D^s\varphi_{N,R}(x) = \sum_{k=-NR}^{NR-1} \delta_{N,k} \left(1+ \frac{k^2}{N^2}\right)^{(s-1)/2}N e^{ikx/N}dn\; .
$$
We recognize a It\^o integral. 

We prove that the sequence $(D^s\varphi_{2^n,R})_n$ is a Cauchy sequence in $\sqrt{1+x^2}L^\infty_{\R},L^2_\Omega$.  

Let $n\geq m$. Given that
$$
\delta_{2^m,l} = \sum_{j=0}^{2^{n-m}-1} \delta_{2^n,2^{n-m} l+j} \; ,
$$
we have
$$
D^s\varphi_{2^n,R}- D^s\varphi_{2^m,R} = \sum_{l=-2^mR}^{2^m R-1} \sum_{j=0}^{2^{n-m}-1} \delta_{2^n,2^{n-m}l+j} \left( \frac{e^{i(2^{n-m}l+j)x/2^n}}{(1+ (2^{n-m}l+j)^2/2^{2n})^{(1-s)/2}} - \frac{e^{ixl/2^m}}{(1+l^2/2^{2m})^{(1-s)/2}}\right)\; .
$$
Taking the $L^2_\Omega$ norm of $D^s\varphi_{2^n,R}- D^s\varphi_{2^m,R}$ to the square, we get
$$
\|D^s\varphi_{2^n,R}- D^s\varphi_{2^m,R}\|_{L^2_\Omega}^2 = \sum_{l=-2^mR}^{2^m R-1} \sum_{j=0}^{2^{n-m}-1} 2^{-n}\Big| \frac{e^{i(2^{n-m}l+j)x/2^n}}{(1+ (2^{n-m}l+j)^2/2^{2n})^{(1-s)/2}} - \frac{e^{ixl/2^m}}{(1+l^2/2^{2m})^{(1-s)/2}}\Big|^2\; .
$$
Since the derivative of $y \mapsto \frac{e^{ixy}}{(1+y^2)^{(1-s)/2}}$ is bounded by $\frac{\sqrt{1+x^2}(1+|s|)}{(1+y^2)^{(1-s)/2}}$, we have that
$$
\Big| \frac{e^{i(2^{n-m}l+j)x/2^n}}{(1+ (2^{n-m}l+j)^2/2^{2n})^{(1-s)/2}} - \frac{e^{ixl/2^m}}{(1+l^2/2^{2m})^{(1-s)/2}}\Big| \leq C_s \frac{\sqrt{1+x^2}}{(1+l^2/2^{2m})^{(1-s)/2}}\frac{j}{2^n}  \; .
$$
As $j$ is less than $2^{n-m}$ we have that $j/ 2^n$ is less than $2^{-m}$ and by summing over $j$, we get
$$
\sum_{j=0}^{2^{n-m}-1} 2^{-n}\Big| \frac{e^{i(2^{n-m}l+j)x/2^n}}{(1+ (2^{n-m}l+j)^2/2^{2n})^{(1-s)/2}} - \frac{e^{ixl/2^m}}{(1+l^2/2^{2m})^{(1-s)/2}}\Big|^2 \leq C_s 2^{-3m}  \frac{1+x^2}{(1+l^2/2^{2m})^{(1-s)}}
$$
Therefore,
$$
\|D^s \varphi_{2^n,R}- D^s\varphi_{2^m,R}\|_{L^2_\Omega}^2\lesssim \sum_{l= -2^m R}^{2^mR-1} 2^{-3m} (1+x^2) (1+ \frac{l^2}{2^{2m}})^{s-1} \lesssim (1+x^2)2^{-2m} \int_{-R}^R\frac{dy}{(1+y^2)^{1-s}}\; .
$$
We get
$$
\|\varphi_{2^n,R}- \varphi_{2^m,R}\|_{L^2_\Omega} \lesssim \sqrt{1+x^2} 2^{-m} \left(\int_{\R} \frac{dy}{(1+y^2)^{1-s}}\right)^{1/2}
$$
hence, since $s<1/2$, the integral converges and $D^s\varphi_{2^n,R}$ is a Cauchy sequence in $\sqrt{1+x^2}L^\infty_\R,L^2_\Omega$ with speed of convergence bounded uniformly in $R$. 

Therefore, the sequence $D^s\varphi_{2^n,R}$ converges uniformly in $R$ in $\sqrt{1+x^2}L^\infty_\R,L^2_\Omega$.\end{proof}

\begin{remark}As it will appear later, the restriction $\sqrt{1+x^2}\chi \in L^1$ will be needed to prove the weak convergence of the finite dimensional measures towards the Gibbs measure because we have to add this weight to have the convergence of $\varphi_{2^n,R}$. \end{remark}

\begin{lemma} The sequence $(D^s\varphi_R)_R$ converges in $\sqrt{1+x^2}L^\infty_\R, L^2_\Omega$. We call its limit $D^s\varphi$. \end{lemma}

\begin{proof} We prove that $\varphi_R$ is a Cauchy sequence. Let $R\geq S$. We have
$$
\varphi_R -\varphi_S = \varphi_R - \varphi_{2^n,R} + \varphi_{2^n,R} - \varphi_{2^n,S} + \varphi_{2^n,S} - \varphi_S\; .
$$
As $D^s\varphi_{2^n,R}$ converges uniformly in $R$ towards $D^s\varphi_R$, it only remains to bound  $D^s\varphi_{2^n,R} - D^s\varphi_{2^n,S}$ independently from $n$. For all $N$, we have 
$$
D^s\varphi_{N,R} - D^s\varphi_{N,S} = \Big( \sum_{k=-NR}^{-NS-1} \delta_{N,k} \left(1+ \frac{k^2}{N^2}\right)^{(s-1)/2} e^{ikx/N} + \sum_{k=NS}^{NR-1} \delta_{N,k} \left(1+ \frac{k^2}{N^2}\right)^{(s-1)/2} e^{ikx/N}\Big)\; .
$$
By taking its $L^2_\Omega$ norm, we get
$$
\|D^s\varphi_{N,R} - D^s\varphi_{N,S}\|_{L^2_\Omega} = \Big( \sum_{k=-NR}^{-NS-1} \frac1N \left(1+ \frac{k^2}{N^2}\right)^{s-1} + \sum_{k=NS}^{NR-1} \frac1N \left(1+ \frac{k^2}{N^2}\right)^{s-1} \Big)^{1/2}\; .
$$
By taking its $L^\infty_\R$ norm, we get
\begin{eqnarray*}
\|D^s\varphi_{N,R} - D^s\varphi_{N,S}\|_{L^\infty_\R,L^2_\Omega} & \lesssim &   \Big( \sum_{k=-NR}^{-NS-1} \frac1N \left(1+ \frac{k^2}{N^2}\right)^{s-1}  + \sum_{k=NS}^{NR-1} \frac1N \left(1+ \frac{k^2}{N^2}\right)^{s-1} \Big)^{1/2}\\
 & \lesssim & \Big( \int_{-R}^{-S} \frac{dy}{(1+y^2)^{1-s}} + \int_{S}^R \frac{dy}{(1+y^2)^{1-s}} \Big)^{1/2}\; .
\end{eqnarray*}
Hence, by taking $N= 2^n$ large enough and $S$ large enough, one can bound $\varphi_R - \varphi_S$ by any $\varepsilon > 0$. The sequence $\varphi_R$ is a Cauchy sequence.
\end{proof}

\begin{proposition}\label{prop-defphi} There exist two sequences $(N_k)_k$ and $(R_k)_k$ that go to $\infty$ when $k\rightarrow \infty$ such that $D^s\phi_k =D^s \varphi_{N_k,R_k}$ converges towards $D^s\varphi$ when $k\rightarrow \infty$ in $\sqrt{1+x^2} L^\infty_\R, L^2_\Omega$. 
\end{proposition}

\begin{proof} We take $R_k = k$ and $N_k = 2^k$. \end{proof}

\paragraph{Spaces which \texorpdfstring{$\varphi$}{phi} belongs to}

In this paragraph, we discuss the support of $\mu$. Indeed, if $\varphi(\omega)$ is $\mathbb P$-almost surely in  a space $X$ then $\mu(X) = \mathbb P(\varphi^{-1}(X)) = 1$, hence the support of $\mu$ is included in $X$. Conversely, if $\varphi(\omega) \notin Y$ $\mathbb P$-almost surely, $\mu(Y) = 0$, the support of $\mu$ is included in $Y^c$, its complementary.

\begin{proposition}\label{prop-belong} For all $1\leq p < \infty$ and $s<1/2$ and $t\in \R$, $L(t)\varphi$ belongs to $ L^p_\Omega,W^{s,p}_{\textrm{loc}}$ and for all $\xi \in L^p_\R$, $\xi L(t)\varphi$ belongs to $L^p_\Omega, L^p_\R$. We also have that for all $1\leq q < \infty$, $\xi L(\tau) \varphi$ belongs to $L^q_{\textrm{loc},\tau }(\R, L^p_x)$ which we write as $L^q_{\textrm{loc},\tau }, L^p_x$ and $L(\tau)\varphi$ belongs to $L^q_{\textrm{loc},\tau},L^p_{\textrm{loc},x}$.  \end{proposition}

\begin{proof}  Let $K$ be a compact of $\R$. The $ L^p_\Omega,W^{s,p}(K)$ norm of $L(t)\varphi$ is equal to the $ L^p_{\textrm{loc},\R}, L^p_\Omega$ norm of $D^s L(t) \varphi$. Let $N, R \in \N$. Since, $L(t) \varphi_{N,R}(x)$ is a Gaussian variable, it $L^p_\Omega$ norm is controlled by its $L^2_\Omega$ norm. Besides, because of the structure of $L(t)$, the $L^2_\Omega$ norm of $L(t)D^s \varphi(x)$ is the same as the $L^2_\Omega$ norm of $D^s \varphi(x)$. Indeed, $\widehat{L(t) f}(n) = e^{-it\sqrt{1+n^2}}\hat f (n)$ and  $|e^{-it\sqrt{1+n^2}}| = 1$. We have
$$
\|D^s L(t) \varphi_{N,R}(x)\|_{L^p_\Omega} \leq C_p  \Big( \sum_{k=-NR}^{NR-1} \frac1N  \left(1+ \frac{k^2}{N^2}\right)^{s-1} \Big)^{1/2}
$$
which yields, for all $x\in \R$,
$$
\|D^s L(t)\varphi_{N,R}(x)\|_{L^p_\Omega} \leq C_p \Big(\int_{\R} \frac{dy}{(1+y^2)^{1-s}} \Big)^{1/2}
$$
which is finite if $s<1/2$. Hence, by taking its $L^p(K)$ norm, we get
$$
\|D^sL(t)\varphi_{N,R}\|_{L^p_K,L^p_\Omega} \leq C_{p,s} \textrm{vol}(K)^{1/p}
$$
where $C_{p,s}$ is a constant depending on $p$ (from the Gaussian) and $s$ (from the integral over $y$) but not on $R$ or $N$. Therefore,
$$
\|D^s L(t) \varphi\|_{L^p_K,L^p_\Omega} \leq C_{p,s} \textrm{vol}(K)^{1/p}\; .
$$

For the second part of the proposition, we use that $\xi$ does not depend on the probability space to write
$$
\|\xi(x) L(t) \varphi_{N,R}(x)\|_{L^p_\Omega} \leq |\xi(x)|\; \|L(t)\varphi_{N,R}(x)\|_{L^p_\Omega}\; ,
$$
and then
$$
\|\xi L(t) \varphi_{N,R}\|_{L^p_\Omega, L^p_\R} \leq  \|\xi\|_{L^p} \|\varphi_{N,R}\|_{L^\infty_\R,L^p_\Omega}\leq C_p \|\xi\|_{L^p}\; .
$$

Finally, let $K$ be a compact of $\R$, and $r$ be the maximum of $p$ and $q$. Thanks to Minkowski inequality, the $L^r_\Omega, L^q_{\tau}(K, L^p_x)$ norm is less than the $L^q_\tau(K, L^p_x), L^r_\Omega$ norm. Since the $L^r_\Omega$ norm of $\xi(x) L(t) \varphi(x)$ does not  depend on $t$, we get
$$
\|\xi L(t) \varphi_{N,R}\|_{L^r_\Omega,L^q_\tau(K, L^p_x)} \leq C_r \textrm{vol}(K)^{1/q}\|\xi\|_{L^p}\; .
$$
We also have with $K'$ another compact of $\R$
$$
\| D^s L(\tau) \varphi_{N,R}\|_{L^r_\Omega,L^q_\tau(K,L^p_x(K'))} \leq C_{r,s} \textrm{vol}(K)^{1/q}\textrm{vol}(K')^{1/p}
$$
which concludes the proof.
\end{proof}

\begin{remark}\label{rem-linfini} The immediate consequence of the previous proposition is that $L(t)\varphi$ belongs almost surely in $\omega \in \Omega$ to $W^{s,p}_{\textrm{loc}} $ and in particular, almost surely to $H^{1/2-}_{\textrm{loc}} = \bigcap_{s<1/2}H^s_{\textrm{loc}}$. Because of Sobolev embeddings, by taking $p > 1/s$, we get that $L(t)\varphi$ belongs almost surely to $L^\infty_{\textrm{loc},x}$ and $L(\tau)\varphi$ belongs to $L^q_{\textrm{loc},\tau},L^\infty_{\textrm{loc},x}$. \end{remark}

\begin{proposition}For $\mathbb P$-almost every $\omega \in \Omega$, $\varphi (\omega)$ does not belong to $L^2_x$. \end{proposition}

\begin{proof} Let $R>0$ and $k\in \N$. We have that 
$$
\int_\Omega \int_{-R}^R |\phi_k(\omega,x)|^2 dx d\mathbb P(\omega) = \int_{-R}^R \int_\Omega |\phi_k(\omega,x)|^2 d\mathbb P (\omega ) dx
$$
and by using the definition of $\phi_k$ :
$$
\int_\Omega \int_{-R}^R |\phi_k(\omega,x)|^2 dx d\mathbb P(\omega) = \int_{-R}^R \sum_{j = -N_kR_k}^{N_kR_k -1} \frac1{N_k} \frac1{1+\frac{j^2}{N_k^2}}dx
$$
We divide the sum in to 3 with $j<0$, $j=0$ and $j>0$ and use symmetries to get
$$
\int_\Omega \int_{-R}^R |\phi_k(\omega,x)|^2 dx d\mathbb P(\omega) \geq 4R \sum_{j=1}^{N_kR_k -1 } \frac{1}{N_k} \frac1{1+\frac{j^2}{N_k^2}} \; .
$$
By comparing the sum with an integral and with a change of variable we get
$$
\int_\Omega \int_{-R}^R |\phi_k(\omega,x)|^2 dx d\mathbb P(\omega) \geq 4R \int_{0}^{R_k - 1/N_k} \frac{dy}{1+y^2} \; .
$$
There exists $k$ such that for all $k \geq k_0$, we have 
$$
\int_\Omega \int_{-R}^R |\phi_k(\omega,x)|^2 dx d\mathbb P(\omega) \geq R\pi \; .
$$
Hence, we get that for all $R$,
$$
\| \left( \int_{-R}^R \varphi(\omega, x) dx\right)^{1/2} \|_{L^2_\Omega} \geq \sqrt{\pi R}
$$ 
and therefore,
$$
\|\varphi\|_{L^2_\Omega,L^2_x} \geq \sqrt{\pi R}\; .
$$
Finally, we have 
$$
\|\varphi\|_{L^2_\Omega, L^2_x} = \infty
$$
and we conclude using Fernique's theorem \cite{Fint}, that states that if $N$ is a norm and $X$ a Gaussian variable (defined in a large sense that includes $\varphi$) then if $\E(N(X)^p)=0$ for any $p >0$, then we have $\mathbb P(N(X) = \infty) = 1$.\end{proof}

\subsection{Finite dimensional approximation and invariance}

\paragraph{Invariance of the law of $\phi_k$ under the linear flow}

\begin{definition}Let $\mu$ be the measure on $H^{1/2-}_{\textrm{loc}}$ induced by $\varphi$ and $\mu_k$ be the measure induced by $\phi_k$. Namely, for all set $A$ in the topological $\sigma$-algebra of $\Hloc$, we have 
\begin{eqnarray*}
\mu(A) = \mathbb P(\varphi^{-1}(A)) \\
\mu_k(A) = \mathbb P(\phi_k^{-1}(A))
\end{eqnarray*}
\end{definition}

\begin{definition}\label{dist-Hloc} For all $R \in \N$ and $s \in \{ \frac12 -\frac1l \; |\; l\in \N^*\}$, we call $p_{R,s}$ the semi-norm :
$$
p_{R,s} : f\mapsto \sqrt{\int_{[-R,R]}|D^s f|^2}\; .
$$
We call $d$ the distance on $\Hloc$ which is equivalent to the topology induced by the semi-norms defined for all $u,v \in \Hloc$ by
$$
d(u,v) = \sum_{(k,l)\in (\N^*)^2} 2^{-(k+l)} \frac{p_{k,1/2-1/l}(u-v)}{1+p_{k,1/2-1/l}(u-v)}\; .
$$
\end{definition}

\begin{proposition}\label{prop-invmuk} The measure $\mu_k$ is invariant under the flow $L(t)$. \end{proposition}

\begin{proof} We apply $L(t)$ to $\phi_k$. We get
$$
L(t) \phi_k(x) = \sum_{j= -N_kR_k}^{N_kR_k-1} e^{i(1+ \frac{j^2}{N_k^2})^{1/2}t}\delta_{N_k,j} \Big( 1+ \frac{j^2}{N_k^2} \Big)^{-1/2} e^{ijx/N_k}\; . 
$$
We recall that $\phi_k$ is given by
$$
\phi_k(x) = \sum_{j= -N_kR_k}^{N_kR_k-1}\delta_{N_k,j} \Big( 1+ \frac{j^2}{N_k^2} \Big)^{-1/2} e^{ijx/N_k}\; . 
$$
The $\delta_{N_k,l}$ are complex Gaussian variables. Hence, their laws are invariant under multiplication by $e^{i\gamma}$ for all $\gamma \in \R$. Besides, they are all independent from each other, hence $L(t) \phi_k$ has the same law as $\phi_k$ (even though they are different random variables), which is equivalent to say that $\mu_k$ is invariant under the flow $L(t)$.
\end{proof}

\paragraph{Approximation of \texorpdfstring{$\mu$}{mu}}

In this paragraph, we give a property on $\mu$ which allows us to approach it by the sequence $(\mu_k)_k$.

\begin{proposition}\label{prop-openmu} For all open set $U$ of $\Hloc$, we have
$$
\mu(U) \leq \liminf_{k\rightarrow \infty} \mu_k(U) \; .
$$
For all closed sets $F$, we have 
$$
\mu(F) \geq \limsup_{k \rightarrow \infty} \mu_k ( F) \; .
$$
\end{proposition}

\begin{proof} This property is equivalent to the fact that $\mu_k$ converges weakly towards $\mu$, which is also equivalent to the fact that for all $F : \Hloc \rightarrow \R$ bounded and Lipschitz-continuous on $\Hloc$, the mean value of $F$ with regard to $\mu_k$ converges towards the mean value of $F$ with regard to $\mu$. In other words, we have to prove that
$$
\E (F \circ \phi_k) \rightarrow \E (F \circ \varphi)
$$
where $\E$ is the mean value with regard to $\mathbb P$.

We have that 
$$
\Big| \E (F \circ \varphi) - \E (F \circ \phi_k)\Big| \leq \E (| F\circ \varphi - F\circ \phi_k |) \leq |F|_{lip} \E (d(\varphi, \phi_k))\; ,
$$
where 
$$
|F|_{lip} = \sup_{x,y \in \Hloc} \frac{|F(x) - F(y)|}{d(x,y)} \; .
$$
Thanks to the definition of $d$, we have 
$$
\E (d(\varphi, \phi_k)) = \sum_{R,l} 2^{-(R+l)} \E \Big(\frac{p_{R,1/2-1/l}(\varphi-\phi_k)}{1+p_{R,1/2-1/l}(\varphi-\phi_k)}\Big)\; .
$$
Since 
$$
\E \Big((\frac{p_{R,1/2-1/l}(\varphi-\phi_k)}{1+p_{R,1/2-1/l}(\varphi-\phi_k)}\Big)
$$
is less than $1$ and than $\E (p_{R,1/2-1/l}(\varphi-\phi_k))$, it is sufficient to prove that the sequences
$$
 \E (p_{R,1/2-1/l}(\varphi-\phi_k))
$$
converge to $0$ when $k\rightarrow \infty$.

For all random variable $f$, all $R\geq 1$ and all $l\geq 1$, we have
$$
\E (p_{R,1/2-1/l}(f) )\leq \| p_{R,1/2-1/l}(f)\|_{L^2_\Omega}
$$
and with $s=1/2-1/l$,
$$
\|p_{R, s} (f)\|_{L^2_\Omega} \leq C R^{3/2}\|(1+x^2)^{-1/2} D^s f \|_{L^\infty_\R,L^2_\Omega}\; .
$$
Hence, $\|p_{R,s}(\varphi - \phi_k)\|_{L^2_\Omega}$ converges toward $0$ when $k \rightarrow \infty$ for all $R$ and $s< 1/2$. Therefore $\mu_k$ converges weakly towards $\mu$ on the topological $\sigma$-algebra of $\Hloc$. 
\end{proof}

\subsection{Invariance of \texorpdfstring{$\mu$}{mu} under the linear flow}

\begin{theorem}\label{th-lininv} The measure $\mu$ is invariant under the flow $L(t)$. That is, for all measurable set $A$ of $\Hloc$,
$$
\mu(L(t)^{-1} A) = \mu(A) \; .
$$
\end{theorem}

\begin{definition} We call $\mu^t$ the image measure of $\mu$ under $L(t)$ and $\mu_k^t$ the image measure of $\mu_k$ under $L(t)$. That is, for all measurable set $A$ of $\Hloc$, we have 
$$
\mu^t ( A) = \mu( L(t)^{-1}(A)) \mbox{ and } \mu_k^t(A) = \mu_k(L(t)^{-1}(A))\; .
$$
\end{definition}

We need the following lemma.

\begin{lemma} The flow $L(t)$ is continuous on $\Hloc$. \end{lemma}

\begin{proof} The kernel of the Klein-Gordon equation (in dimension one) is given by
$$
K_t(z) = -\frac1{2\pi}\int dk e^{i\sqrt{1+k^2}t}e^{ikx} 
$$
and is equal to $0$ if $z^2\leq t^2$, see \cite{Zeenutshell,PSintro}.  What we want to point out is that since the Klein-Gordon equation has a finite propagation speed, we get that the value of $L(t) u$ in $x$ depends only on the values of $u$ in $[x-|t|,x+|t|]$. We have 
$$
D^s L(t) u(x) = L(t) D^s u(x) =  \int dy K_t(x-y) (D^s u)(y) 
$$
and $K_t(x-y)u(y) = 1_{|x-y|\leq |t|} K(x-y) u(y)= K_t(x-y) 1_{|y| \leq |x| + |t|} u(y)$ where $1$ is the indication function. Hence if $|x| \leq R$, we get
$$
K_t(x-y) (D^s u)(y) = K_t(x-y)(D^s u)_{R + |t|} (y)\; ,
$$
where $u_{R+|t|}$ is the restriction of $u$ to $[-R+|t|,R+|t|]$ and finally
$$
L(t) u(x) = L(t)u_{R+|t|} (x)\; .
$$
We take its $p_{R,s}$ semi norm, we get
$$
p_{R,s} (L(t) u) = p_{R,0} (L(t)(D^s u)_{R+|t|}) \leq \|L(t) (D^s u)_{R+|t|}\|_{L^2} = p_{R+|t|,s}(u) \; .
$$
Therefore, $L(t)$ is continuous in $\Hloc$.
\end{proof}

We now prove Theorem \ref{th-lininv}.

\begin{proof} Let $ K $ be a closed set of $\Hloc$. We call $K_\varepsilon$ the set
$$
K_\varepsilon = \{u \in \Hloc \; |\; \exists v \in K \, ;\, d (u,v) < \varepsilon\} \; .
$$
The set $K_{\varepsilon}$ is open in $\Hloc$ , and since $L(t)$ is continuous in $\Hloc$, so is the set $L(t)^{-1}(K_\varepsilon)$. Hence, we have, thanks to Proposition \ref{prop-openmu}
$$
\mu^t(K) \leq \mu^t(K_{\varepsilon}) \leq \liminf_{k\rightarrow \infty} \mu_k^t (K_{\varepsilon})\; .
$$
Since $\mu_k$ is invariant under the flow $L(t)$ (Proposition \ref{prop-invmuk}), we have 
$$
\mu^t(K) \leq \liminf_{k\rightarrow \infty} \mu_k (K_{\varepsilon})\; .
$$
Since $\liminf$ is less than $\limsup$ and a set is included in its closure, we have 
$$
\mu^t(K) \leq \limsup_{k\rightarrow \infty} \mu_k ( \overline K_{\varepsilon})\; ,
$$
with
$$
\overline K_{\varepsilon} = \{u \in \Hloc \; |\; d(u,K) \leq \varepsilon \}\; ,
$$
Where $d(u,K) = \inf_{v\in K} d(u,v)$.

As the set $ \overline K_{\varepsilon}$ is closed we have, thanks again to Proposition \ref{prop-openmu}
$$
\mu^t(K) \leq \mu ( \overline K_{\varepsilon})\;.
$$
We deduce from that
$$
\mu^t(K) \leq \inf_{\varepsilon > 0}  \mu (  \overline K_{\varepsilon})\; .
$$
 Finally, thanks to the dominated convergence theorem when $\varepsilon \rightarrow 0$, as $K$ is closed
$$
\mu^t (K) \leq \mu (K)\; .
$$
We use the reversibility of the flow to conclude. Indeed, as $L(t)$ is continuous for the topology of $\Hloc$, $L(t)^{-1}(K)$ is closed and hence
$$
\mu(K) = \mu^{-t} (L(t)^{-1} K) \leq \mu (L(t)^{-1}K) = \mu^t (K)\; .
$$
Hence for all closed set $K$, $\mu^t(K) = \mu(K)$. We get the invariance on all measurable sets using that the closed sets of $\Hloc$ generate its topological $\sigma$-algebra.
\end{proof}

\section{Definition and approximation of the Gibbs measure}

We define here the measure $\rho$ which we will prove to be invariant under the flow of \eqref{NLKG} and a sequence of measures $\rho_k$ which we use to prove the invariance of $\rho$.

\subsection{Definition of \texorpdfstring{$\rho_k$}{p k} and invariance under \texorpdfstring{$\psi_k$}{psi k}}

In this subsection, we prove the invariance of the measures $\rho_k$, which approach the Gibbs measure $\rho$, under the flows $\psi_k(t)$, which we will use later as an approximation of the flow of the cubic Klein-Gordon equation.

\begin{definition}We call $f_k$ the function on the support of $\mu$ defined as :
$$
f_k (u) = e^{-\frac1{4\pi}\int_{-\pi N_k}^{\pi N_k} \chi |\re u|^{4}}\; .
$$
It belongs to $L^1_{\mu_k}$. 
We call $f$ the function on the support of $\mu$ such that :
$$
f (u) = e^{-\frac1{4\pi}\int_{\R} \chi |\re u|^{4}}\; .
$$
It belongs to $L^1_\mu$.
We call $\rho$ and $\rho_k$ the measures on $\Hloc$ defined as 
$$
d\rho(u) = \frac{f(u)}{\Gamma} d\mu(u) \; , \; d\rho_k(u)  =\frac{f_k(u)}{\Gamma_k} d\mu_k(u) \; ,
$$
where $\Gamma= \|f\|_{L^1_\mu}$ and $\Gamma_k = \|f_k\|_{L^1_{\mu_k}}$ are normalisation factors.
\end{definition}

We approach \eqref{NLKG} by a finite-dimensional equation.

Let $E_k$ be the set of trigonometrical polynomials :
$$
E_k = \textrm{Vect} \left( \lbrace x\mapsto e^{ilx/N_k} \; |\; l=-N_kR_k ,\hdots ,N_kR_k-1 \rbrace \right) \; .
$$

Let $P_k$ be the operator on $\Hloc$, such that for all $f\in \Hloc$,
$$
P_k f(x) = f \Big( x - \ent{\frac{x+ \pi N_k}{2\pi N_k}}2\pi N_k\Big) \; .
$$
Let us note that $P_k f ( x) = f(x)$ for all $x\in [-\pi N_k,\pi N_k[$ and that $P_k f$ is $2\pi N_k$ periodic. 

Finally, let $\Pi_k$ be the Fourier multiplier :
$$
\widehat{\Pi_k f} (n) = 1_{[-R_k, R_k[}(n) \hat f(n) \; ,
$$
where $1_{[-R_k,R_k[} (n)$ is equal to $1$ if $n \in [-R_k,R_k[$ and $0$ otherwise.

The finite-dimensional equation with which we approach \eqref{NLKG} is 
\begin{equation}\label{finitedim}
i\partial_t u -\sqrt{1-\lap} u = (1-\lap)^{-1/2}\Pi_k P_k \Big( \chi(\re u)^{3} \Big)\: .
\end{equation}

\begin{proposition}\label{prop-hamil}
The equation \eqref{finitedim} is a Hamiltonian equation on $E_k$ with Hamiltonian
$$
H_k(u) =  \frac1{2\pi}\int_{-\pi N_k}^{\pi N_k} \overline u (1-\lap)u + \frac1{4\pi} \int_{-\pi N_k}^{\pi N_k} \chi |\re u|^{4} \; .
$$\end{proposition}

\begin{proof} Let $H_k(u) = H_c (u) + H_p (u)$ with
$$
H_c(u) =  \frac{1}{2\pi}\int_{-\pi N_k}^{\pi N_k} \overline u (1-\lap)u \mbox{ and } H_p(u) =  \frac1{4\pi} \int_{-\pi N_k}^{\pi N_k} \chi |\re u|^{4} \; .
$$
The quantity $H_c(u)$ represents the kinetic energy, and $H_p(u)$ the potential energy.

For all $u\in E_k$, we write
$$
u = \sum_{j= -N_k R_k}^{N_k R_k -1} u_j e^{ijx/N_k} \; .
$$
We can write $H_c(u)$ as
$$
H_c(u) =  N_k \sum_{j=-N_k R_k }^{N_k R_k -1}  (1+\frac{j^2}{N_k^2})|u_j|^2\; .
$$
Therefore, 
$$
\frac{dH_c(u)}{du_l} =  N_k (1+ \frac{l^2}{N_k^2})\overline{u_l}\; .
$$
We can write $H_p(u)$ as 
$$
H_p(u) = \frac12 N_k \sum_{j_1+j_2+j_3+j_4+j_5 = 0} (P_k \chi)_{j_1}(\re u)_{j_2}(\re u)_{j_3}(\re u)_{j_4}(\re u)_{j_5} \; .
$$
Since $(\re u)_j = \frac12 (u_j + \overline{u_{-j}})$, we have $\frac{d(\re u)_j}{du_l} = \frac12 \delta_j^l $. Therefore,
$$
\frac{dH_p(u)}{du_l} = \pi N_k \frac42 \sum_{j_1+j_2+j_3+j_4 = -l} (P_k \chi)_{j_1}(\re u)_{j_2}(\re u)_{j_3}(\re u)_{j_4} = 2\pi N_k [(P_k \chi) (\re u)^3]_{-l}\; .
$$
Besides, if $v$ is real then, $\overline{v_l}= v_{-l}$ and the equation \eqref{finitedim} can be written as
$$
-i\partial_t \overline u = (1-\lap)^{-1/2} \Big( (1-\lap) \overline u + \overline{\Pi_k [(P_k \chi) (\re u)^3] }\Big)
$$
and in terms of $u_l$ as
$$
-i\partial_t \overline{u_l}  = (1+\frac{l^2}{N_k^2})^{-1/2} \Big( (1+\frac{l^2}{N_k^2}) \overline{u_l} + \overline{[(P_k \chi) (\re u)^3]_l}\Big) = (1+\frac{l^2}{N_k^2})^{-1/2} \frac1{2\pi N_k} \frac{dH_k(u)}{du_l}
$$
or in other words,
$$
\partial_t \overline u = J \bigtriangledown_u H_k
$$
with 
$$
J = i \frac1{N_k} (1-\lap)^{-1/2}
$$
a anti-Hermitian operator. Hence, \eqref{finitedim} is a Hamiltonian equation with Hamiltonian $H_k$.\end{proof}

\begin{proposition} The equation \eqref{finitedim} is globally well-posed on $E_k$. We call its flow $\psi_k$.\end{proposition}

\begin{proof} The space $E_k$ is of finite dimension and the equation has a locally Lipschitz-continuous non-linearity. The Hamiltonian $H_k$ controls the $p_{\pi N_k,1}$ semi-norm, which is a norm on $E_k$. \end{proof}

\begin{proposition}\label{prop-rhok} The measure $\rho_k$ is invariant under the flow $\psi_k(t)$. \end{proposition}

\begin{proof} The flow $\psi_k$ is Hamiltonian, hence the Lebesgue measure is invariant under $\psi_k$ thanks to Liouville's theorem. Thanks to Remark \ref{rem-Glaw}, we have that
$$
d\mu_k (u)  = d_k e^{-H_c (u)} dL(u)
$$
where $d_k$ is a normalisation factor and $L$ is the Lebesgue measure. Therefore,
$$
d\rho_k(u) = \frac{f_k(u)}{\Gamma_k}d_k e^{-H_c (u)} dL(u) = \frac{d_k}{\Gamma_k} e^{-H_k(u)}dL(u)\; .
$$
Since $H_k(u)$ is a Hamiltonian for the equation, it is invariant under the flow of the equation, hence the measure $\rho_k$ is invariant under $\psi_k(t)$. \end{proof}

\subsection{Approximation of \texorpdfstring{$\rho$}{p}}

In this subsection, we prove that we can approach $\rho$ by $\rho_k$.

\begin{proposition}\label{prop-openrho} For all open set $U$ of $\Hloc$, we have 
$$
\rho (U) \leq \liminf_{k \rightarrow \infty} \rho_k(U) \; ,
$$
and for all closed set $C$, we have 
$$
\rho(C) \geq \limsup_{k\rightarrow \infty} \rho_k(C)\; .
$$
\end{proposition}

To prove this proposition, we prove first some lemmas.

\begin{lemma}The sequence $\|f-f_k\|_{L^1_{\mu_k}}$ converges towards $0$ when $k$ goes to $\infty$. the sequence $(f_k)_k$ converges towards $f$ in $L^1_\mu$.
\end{lemma}

\begin{proof} Let $J_k  = \R \smallsetminus [-\pi N_k, \pi N_k]$. We have for all $u$
$$
\Big| f(u) -f_k(u)\Big| = \frac1{4\pi } \int_{J_k} \chi |u|^{4}\; .
$$
By taking its $L^1_{\mu_k}$ norm, we get
$$
\|f_k -f\|_{L^1_{\mu_k}} \leq \int_{J_k} \chi \|u\|_{L^{4}_{\mu_k}}^{4} \leq \|u\|_{L^\infty_x,L^{4}_{\mu_k}}^{4} \int_{J_k} \chi \; .
$$
We recall from the proof of Proposition \ref{prop-belong} that
$$
\|u\|_{L^\infty_x,L^{4}_{\mu_k}} = \|\phi_k\|_{L^\infty_x,L^{4}_{\Omega}} \leq C \Big( \int \frac{dy}{1+y^2} \Big)^{1/2}
$$
hence it is bounded independently from $k$. 

Finally, we have 
$$
\|f_k -f\|_{L^1_{\mu_k}} \leq C \int_{J_k} \chi
$$
and since $\chi \in L^1$, $\|f_k -f\|_{L^1_{\mu_k}}$ goes to $0$ when $k\rightarrow  \infty$.

For the same reasons, simply replacing $L^4_{\mu_k}$ by $L^4_\mu$, we have
$$
\|f_k -f\|_{L^1_{\mu}} \leq C \int_{J_k} \chi
$$
and thus converges towards $0$ when $k$ goes to $\infty$.
\end{proof}

\begin{lemma} The sequence $(f\circ \phi_k)_k$ converges towards $f\circ \varphi$ in $L^1_\Omega$.\end{lemma}

\begin{proof}For all $\omega$, we have 
$$
\Big| f\circ \phi_k (\omega) - f\circ \varphi(\omega) \Big| \lesssim \int \chi | \phi_k(\omega) - \varphi(\omega)| (|\phi_k(\omega)|^3+|\varphi(\omega)|^3)\; .
$$
By taking its $L^1_\Omega$ norm, we get
$$
\|f\circ \phi_k - f\circ \varphi\|_{L^1_\Omega} \lesssim \int \chi  \| \phi_k - \varphi\|_{L^{4}_\Omega} (\|\phi_k(\omega)\|_{L^{4}_\Omega}^3+ \|\varphi(\omega)\|_{L^{4}_\Omega}^3)\; .
$$
We use the facts that $\|\phi_k\|_{L^\infty_x, L^{4}_\Omega}$ is uniformly bounded in $k$, that $\phi_k$ converges towards $\phi$ in \\
$\sqrt{1+x^2}L^\infty_x, L^{4}_\Omega$, and that $\chi \sqrt{1+x^2}$ belongs to $L^1$ to conclude.\end{proof}

We prove Proposition \ref{prop-openmu}. For that, we prove that $\rho_k$ converges weakly towards $\rho$.
\begin{proof}
Let $F$ be a bounded Lipschitz-continuous function. We have
$$
\E_\rho(F) - \E_{\rho_k}(F) = \int F(u) \frac{f(u)}{\Gamma} d\mu(u) - \int F(u) \frac{f_k(u)}{\Gamma_k} d\mu_k(u) = I +II
$$
with 
$$
I  = \int F(u) \Big( \frac{f(u)}{\Gamma} - \frac{f_k(u)}{\Gamma_k}\Big) d\mu_k (u)
$$
and
\begin{eqnarray*}
II&=&  \int F(u) \frac{f(u)}{\Gamma}d\mu(u) - \int F(u) \frac{f(u)}{\Gamma} d\mu_k(u) \\
 &= & \int \Big( F\circ \varphi(\omega) \frac{f\circ \varphi(\omega)}{\Gamma} - F\circ \phi_k(\omega) \frac{f\circ \phi_k(\omega)}{\Gamma}\Big) d\mathbb P(\omega) \; .
\end{eqnarray*}

For $I$, we use that $F$ is bounded, thus
$$
|I| \leq \|F\|_{L^\infty} \|\frac{f}{\Gamma} - \frac{f_k}{\Gamma_k}\|_{L^1_{\mu_k}}\; .
$$
Then, $\Gamma = \|f\|_{L^1_\mu}$ and $\Gamma_k = \|f_k\|_{L^1_{\mu_k}}$. Hence,
$$
|\Gamma-\Gamma_k| \leq \|f\|_{L^1_\mu}- \|f\|_{L^1_{\mu_k}} + \|f-f_k\|_{L^1_{\mu_k}}
$$
and besides 
$$
 \|f\|_{L^1_\mu}- \|f\|_{L^1_{\mu_k}} \leq \|f\circ \phi_k - f\circ \varphi\|_{L^1_\Omega}\; .
$$
Thanks to the previous lemmas, $\Gamma_k \rightarrow \Gamma$. Besides, $\|f-f_k\|_{L^1_{\mu_k}}$ converges towards $0$. Therefore, $I$ goes to $0$ when $k\rightarrow \infty$.

We write $II$ as $II = II.1 + II.2$ with
$$
II.1  = \int F \circ \varphi(\omega) \frac{f\circ \varphi(\omega) - f\circ \phi_k(\omega)}{\Gamma} d\mathbb P(\omega) 
$$
and 
$$
II.2 = \int \frac{f\circ \phi_k(\omega)}{\Gamma}\Big( F \circ \varphi(\omega) - F \circ \phi_k (\omega) \Big) d\mathbb P(\omega) \; .
$$
For $II.1$ we use that $F$ is bounded, hence
$$
|II.1| \leq \|F\|_{L^\infty} \frac1\Gamma \|f \circ \varphi - f\circ \phi_k\|_{L^1_\Omega}
$$
which ensures that $II.1$ goes to $0$ when $k \rightarrow \infty$.

For $II.2$, we use first that $f$ is less than $1$, hence
$$
|II.2| \leq \frac1\Gamma \int \Big| F\circ \phi_k (\omega) - F \circ \varphi(\omega) \Big| d\mathbb P(\omega)\; .
$$
Then, as $F$ is Lipschitz-continuous, we have that 
$$
\Big|  F\circ \phi_k (\omega) - F \circ \varphi(\omega) \Big| \leq C d(\varphi(\omega), \phi_k(\omega)) \; .
$$

We previously proved that $\E(d(\varphi,\phi_k))$ converged towards $0$ (this is used in the proof of the weak convergence of $\mu_k$ towards $\mu$), hence $\rho_k$ converges weakly towards $\rho$ which is equivalent to : for all open sets $U$,
$$
\rho(U) \leq \liminf_{k\rightarrow \infty} \rho_k(U) \; ;
$$
or to : for all closed sets $C$,
$$
\rho(C) \geq \limsup_{k\rightarrow \infty } \rho_k(C) \; .
$$ \end{proof}

\section{Local properties and approximation of \texorpdfstring{$\psi$}{psi}}

In this section, we prove the local posedness of $\psi$. Then, we build an extension of $\psi_k$ to the whole $\Hloc$, prove that it is globally well-posed, and that the sequence $\psi_k(t)$ converges locally in time, uniformly in some sets of $\Hloc$, towards $\psi$, the flow of the cubic KG equation. Then, we prove that the flows $\psi(t)$ and $\psi_k(t)$ are continuous with regard to the initial datum, locally in time for $\psi$ and globally for $\psi_k$.

\subsection{Local well-posedness of \texorpdfstring{$\psi$}{psi}}

In this subsection, we prove that the flow is locally well-posed in some sets.

In the rest of the paper $L^p_\tau$ denotes the $L^p([-1,1])$ space in time. When we use the letter $t$ instead of $\tau$, we either precise the support of integration, or we consider that this support is $\R$. 

\begin{proposition}\label{prop-lwp} There exists $C\geq 1$ such that for all $\Lambda \geq 1$ and all $\varepsilon > 0$, with $T = \frac{\varepsilon^2}{C \Lambda^{4}}$, for all $u_0\in \Hloc$ such that
$$
\|\chi^{1/3} | L(\tau) u_0|\|_{L^6_\tau,L^6_x} \leq \Lambda
$$
then the Cauchy problem 
\begin{equation}\label{KGcomp} \left \lbrace{\begin{tabular}{ll}
$i\partial_t u - \sqrt{1-\lap} u = (1-\lap)^{-1/2} \chi (\re u)^3 $ \\
$u_{|t=0} = u_0$ \end{tabular}} \right.
\end{equation}
has a unique solution in $\mathcal C ([-T,T],H^1) + L(t) u_0$. We call the flow of the equation $\psi$. We have
$$
\|\psi(t)u_0 - L(t)u_0\|_{H^1} \leq \varepsilon \Lambda \; .
$$ 
\end{proposition}

\begin{proof} The problem \eqref{KGcomp} is equivalent to the problem on $v = u - L(t)u_0$ : 
\begin{equation} \label{KGonv}
\left \lbrace{\begin{tabular}{ll}
$i\partial_t v- \sqrt{1-\lap} v  = (1-\lap)^{-1/2} \chi (\re (L(t)u_0 + v))^{3} $ \\
$v_{|t=0} = 0$ \end{tabular}}\right. .
\end{equation}
The Duhamel formulation of this equation is 
$$
v (t) = A(v)(t) = -i \int_{0}^t L(t-\tau)(1-\lap)^{-1/2} \chi (\re (L(\tau)u_0 + v(\tau)))^{3} d\tau \; .
$$
We use a contraction argument on $A$. 

For all $t\in [-1,1]$, we have 
$$
\|A(v)(t)\|_{H^1} \leq \int_{0}^t \|L(t-\tau)(1-\lap)^{-1/2} \chi (\re (L(\tau)u_0 + v(\tau)))^{3} \|_{H^1}d\tau \; .
$$
Hence,
$$
\|A(v)(t)\|_{H^1} \leq \int_{0}^t \| \chi |\re (L(\tau)u_0 + v(\tau))|^{3}\|_{L^2}d\tau \; .
$$
And with a triangle inequality, we get
$$
\|A(v)(t)\|_{H^1} \leq \int_{0}^t \Big( \|\chi |L(\tau)u_0|^3\|_{L^2} + \|\chi\|_{L^2}\|v(\tau)\|_{L^\infty}^3 \Big) d\tau\; .
$$
Since the dimension is $1$, $L^\infty$ is embedded into $H^1$ and thus, for $T \leq 1$,
$$
\|A(v)\|_{L^\infty([-T,T],H^1)} \lesssim T^{1/2} \Lambda^3 + T \|v\|_{L^\infty([-T,T],H^1)}^3\; .
$$
Hence, if $\|v\|_{L^\infty([-T,T],H^1)}$ is less than $\varepsilon\Lambda$ and $T$ is less than $\frac{\varepsilon^2}{C\Lambda^{4}}$ with $C$ big enough, we have 
$$
\|A(v)\|_{L^\infty([-T,T],H^1)} \leq \Lambda \varepsilon\; .
$$

The ball of $L^\infty([-T,T],H^1)$ of radius $\varepsilon\Lambda$ is stable under $A$.

Let $v$ and $w$ in this ball. We have 
\begin{eqnarray*}
\|A(v)(t) - A(w)(t)\|_{H^1}  &\lesssim & \int_{0}^t \Big( \|\chi^{1/3}|L(\tau)u_0|\|_{L^{6}}^{2} + \|\chi^{1/3}|v|\|_{L^{6}}^{2}+ \|\chi^{1/3}|w|\|_{L^{6}}^{2}\Big) \|v-w\|_{L^{6}} \\
 &\lesssim & \Big( |t|^{1/2+1/6} \Lambda^{2} + |t| \Lambda^{2} \Big) \|v-w\|_{L^\infty([-T,T],H^1)}
\end{eqnarray*}
which makes $A$ contracting as long as $T \leq \frac{\varepsilon^2}{C\Lambda^{4}}$ with $C$ big enough.

Therefore, there exists a unique fixed point of $A$ and thus a unique solution of \eqref{KGcomp} in $L(t)u_0 + \mathcal C ([-T,T],H^1)$. \end{proof}

\subsection{Extension of \texorpdfstring{$\psi_k$}{psi k} to \texorpdfstring{$\Hloc$}{H loc 1/2-}}\label{subsec-exten}

In order to approach the flow $\psi$ by the flows $\psi_k$, we need to extend $\psi_k$ from $E_k$ to the support of $\mu$, and hence to the support of $\rho$.

To extend the flow $\psi_k(t)$ to the support of $\mu$, we define it as the flow of the equation
\begin{equation}\label{app}
i\partial_t u - \sqrt{1-\lap} u = (1-\lap)^{-1/2}\Pi_k P_k \Big( \chi (\re u)^3\Big)\; .
\end{equation}

\begin{proposition} There exists $C\geq 1$ such that for all $k$, all $\Lambda \geq 1$, all $\varepsilon > 0$ and with $T = \frac{\varepsilon^2}{C \Lambda^{4}}$, for all $u_0$ such that
$$
\|\chi | L(\tau) u_0|^{3}\|_{L^2_\tau([t_0-1,t_0+1],L^{2}([-\pi N_k, \pi N_k]))} \leq \Lambda^3
$$
and all $v_0\in E_k$ such that $p_{\pi N_k, 1}(v) \leq \varepsilon \Lambda$ then there exists a unique solution of 
$$
\left \lbrace{ \begin{tabular}{ll}
$i\partial_t v = \sqrt{1-\lap} v + (1-\lap)^{-1/2} \Pi_k P_k (\chi \re(L(t)u_0+v))$ \\
$v_{|t=t_0} = v_0 $\end{tabular}} \right.
$$
in $\mathcal C([-T,T],E_k)$, and this solution satisfies 
$$
\|p_{\pi N_k, 1}(v)\|_{L^\infty([-T,T])} \leq 3\varepsilon \Lambda\; .
$$\end{proposition}

\begin{proof}In order to generate the same kind of estimate that we found for the equation \eqref{KGcomp}, we repeat the argument. The equation on $v$ can be rewritten in the form : 
$$
v(t) = L(t-t_0) v_0 + A_k(v)(t) =  L(t-t_0)v_0 - i \int_{t_0}^t L(t-\tau) (1-\lap)^{-1/2} P_k \Pi_k \chi_k (\re(L(\tau)u_0 +v(\tau) )^{3}  \; .
$$
We can prove that $L(.-t_0)v_0+A_k$ is contracting on the ball of radius $3\varepsilon \Lambda$ of $\mathcal C [-T,T], E_k$ normed with $L^\infty_t, p_{\pi N_k,1}$. We use the same argument as previously along with some remarks. The first one is that even though the linear flow $L(t)$ shifts the spatial support of the norm, we consider it only for finite time and hence, with $f \in \Hloc$, we have, since $|t-\tau| \leq T \leq 1 \leq \pi N_k$
$$
p_{\pi N_k,1} (L(t-\tau) f ) \leq p_{2\pi N_k, 1} (f) \leq 2 p_{\pi N_k, 1}(f) \; .
$$
The second remark is that $\Pi_k$ satisfies 
$$
p_{\pi N_k,0}(\Pi_k f) \leq p_{\pi N_k, 0}(f)\; .
$$
The third one is that $P_k f$ restricted to $[-\pi N_k, \pi N_k[$ is equal to the restriction of $f$ hence
$$
\| P_k \chi | L(\tau) u_0|^{3}\|_{L^2_\tau,L^{2}([-\pi N_k,\pi N_k])} \leq \|\chi | L(\tau) u_0|^{3}\|_{L^2_\tau,L^{2}([-\pi N_k, \pi N_k])} \leq \Lambda\; .
$$
Finally, we have the Sobolev inequality in $E_k$,
$$
\|f\|_{L^\infty([-\pi N_k,\pi N_k])} \leq C p_{\pi N_k, 1} (f) 
$$
with a constant $C$ independent from $N_k$. (We recall that if $f$ is in $E_k$ then $f$ is $2\pi N_k$ periodic.)

The proof follows the same steps as the one of Proposition \ref{prop-lwp}.\end{proof}

\begin{proposition}\label{prop-gwp} The flow $\psi_k$ is globally well-posed in $L(t) u_0+E_k$ on the support of $\mu$. \end{proposition}

\begin{proof} We use energy estimates. Let 
$$
\mE (c) = \frac12 \int_{-\pi N_k}^{\pi N_k} \overline v (1-\lap) v + \frac1{4}  \int_{-\pi N_k}^{\pi N_k} \chi |\re v|^{4} \; .
$$
Its derivative in time is given by
$$
\partial_t \mE (v) = \re \Big(  \int_{-\pi N_k}^{\pi N_k} \Big( (1-\lap) \overline v + \chi(\re v)^{3}\Big) \partial_t v \Big)
$$
and by inserting the formula for $\partial_t v$,
$$
\partial_t \mE (v) = \re \Big( (-i) \int_{-\pi N_k}^{\pi N_k} \Big( (1-\lap ) \overline v + \chi (\re v)^3\Big)  \Big(\sqrt{1-\lap} v + (1-\lap)^{-1/2}  \Pi_k P_k (\chi (\re L(t)u_0 + v)^3)\Big)\Big) \; .
$$
Because of the real part in this expression, we keep only the following terms
$$
\partial_t \mE (v) = \re \Big( (-i) \int_{-\pi N_k}^{\pi N_k} \sqrt{1-\lap} \overline v \chi \Big( (\re (L(t)u_0 + v))^3 - (\re v)^3 \Big) \; .
$$
Therefore, using Cauchy-Schwartz inequality, 
$$
|\partial_t \mE (v) |\lesssim \left(\int_{-\pi N_k}^{\pi N_k} \overline v (1-\lap )v \right)^{1/2} p_{\pi N_k, 0}\left(\chi \Big( (\re (L(t)u_0 + v))^3 - (\re v)^3 \Big) \right)\; .
$$
First, we remark that
$$
 p_{\pi N_k,0}(\chi \Big( (\re (L(t)u_0 + v))^3 - (\re v)^3 \Big)) \lesssim  p_{\pi N_k, 0}\left(\chi \Big( |L(t)u_0| ( |\re v|^2+ |L(t) u_0|^2)\right)\; .
$$
From the first section, Remark \ref{rem-linfini}, we have that $L(t) u_0$ belongs almost surely to $L^6_{[-\pi N_k, \pi N_k]}$ , hence we get
$$
p_{\pi N_k, 0}\left(\chi \Big( |L(t)u_0| ( |\re v|^2+ |L(t) u_0|^2)\right) \leq \|L(t) u_0\|_{L^6_{[-\pi N_k, \pi N_k]}}^3 + \|L(t) u_0\|_{L^6_{[-\pi N_k, \pi N_k]}} \|\chi^{1/2}\re v\|_{L^6_{[-\pi N_k, \pi N_k]}}^2   \; .
$$
Using that  $\re v$ belongs to $E_k$ which is of finite dimension, we get that 
$$
\| \chi^{1/2} \re v\|_{L^6_{[-\pi N_k, \pi N_k]}} \leq C_k \|\chi^{1/2} \re v\|_{L^4_{[-\pi N_k, \pi N_k]}}\; .
$$
Indeed, we have 
$$
\| \chi^{1/2} \re v\|_{L^p_{[-\pi N_k, \pi N_k]}} = \| P_k \chi^{1/2} \re v\|_{L^p_{[-\pi N_k, \pi N_k]}}
$$
and besides, we can build a set of finite dimension $\widetilde E_k$ in the following way :
$$
P_k \chi^{1/2} \re v \in \widetilde E_k : = \mbox{Vect}\Big(P_k \chi^{1/2} e^{ijx/N_k} \; , \; j= -R_k N_k , \hdots , N_k R_k -1 \Big)
$$
and this space $\widetilde E_k$ is normed both by $L^4([-\pi N_k, \pi N_k])$ and by $L^6([-\pi N_k, \pi N_k])$.

We note that
$$
\int_{-\pi N_k}^{\pi N_k} \overline v (1-\lap )v \mbox{ and }  \|\chi^{1/2} \re v\|_{L^4_{[-\pi N_k, \pi N_k]}}^4 
$$
are less than $\mE(v)$. We get with $\mH (v) = \sqrt{\mE (v)}$,
$$
\partial_t \mH (v) \lesssim \|L(t) u_0\|_{L^6_{[-\pi N_k, \pi N_k]}} \Big( \mH(v) + \| L(t) u_0 \|_{L^6([]-\pi N_k,\pi N_k)}^2\Big) \; .
$$
Using Gronwall lemma, and the fact that the initial datum is $0$, we get
$$
\mH(v) \lesssim \int_{0}^t \|L(\tau) u_0\|_{L^6_{[-\pi N_k, \pi N_k]}}^3 d\tau \; e^{\int_{0}^t  \|L(\tau) u_0\|_{L^6_{[-\pi N_k, \pi N_k]}} d\tau} \; .
$$
Therefore, since for $u_0$ the support of $\mu$, 
$$
\int_{0}^t  \|L(\tau) u_0\|_{L^6_{[-\pi N_k, \pi N_k]}}^3 d\tau \mbox{ and }, \int_{0}^t \|L(\tau)u_0\|_{L^6} d \tau 
$$
are finite for all $t$, $\mH(v)$ remains finite at all time. It controls the $p_{\pi N_k, 1}$ norm of $v$, hence this norm remains finite. What is more, in the support of $\mu$,  the $L^2([t_0-1,t_0+1], L^2_x$ of $\chi (L(\tau) u_0)^3$ is finite and uniformly bounded in $t_0$ in compacts of $\R$. This yields the global well posedness of \eqref{app}. \end{proof}

\subsection{Local uniform convergence of \texorpdfstring{$\psi_k$}{psi k} towards \texorpdfstring{$\psi$}{psi}}

In this subsection, we prove the uniform convergence of $\psi_k$ towards $\psi$.

\begin{proposition}\label{prop-luc}There exists $C$ such that for all $\Lambda \geq 1$, all $q\geq 9$, all $\kappa \in [6,12]$ and $T = \frac{1}{C \Lambda^{4}}$, we have that for all $\varepsilon > 0$, all $R>0$, there exists $k_0$ such that for all $k \geq k_0$, all $|t| \leq T$ and all $u_ 0 \in A$ with
$$
A = \lbrace u \; \Big|  \|\chi^{1/\kappa} |L(t)u|\|_{L^{q}_\tau,L^{q}_x}\leq \Lambda \rbrace \; ,
$$
we have 
$$
\|  (\sqrt{1+x^2})^{-\alpha}D^{3/4}(\psi(t)u_0 - \psi_k(t)u_0)\|_{L^2(\R)} \leq \varepsilon \; .
$$
\end{proposition}

Remark that even though $\psi(t) u_0$ is not in $H^{3/4}$ with a weight, $\psi(t)u_0-L(t) u_0$ is and so is $\psi_k(t)u_0 - L(t)u_0$.

\begin{lemma} For all $R \geq 1$ and all $u_0 \in A$ and $|t|\leq T$, we have 
$$
\|  (\sqrt{1+x^2})^{-\alpha}D^{3/4}(\psi(t)u_0 - \psi_k(t)u_0)\|_{L^2(\R)} \leq C(R) \Lambda +  p_{R,3/4}(\psi(t)u_0 - \psi_k(t)u_0)
$$
Where $C(R)$ goes to $0$ when $R$ goes to $\infty$. \end{lemma}

\begin{proof}We look at what happens for $|x|\geq R$. Then,
$$
\| 1_{|x|\geq R} (\sqrt{1+x^2})^{-\alpha}D^{3/4}(\psi(t)u_0 - L(t)u_0)\|_{L^2} \leq (\sqrt{1+R^2})^{-\alpha} \|\psi(t)u_0 - L(t)u_0\|_{H^1} \lesssim \sqrt{1+R^2}^{-\alpha} \Lambda
$$
for the desired times.

Then, we divide $\{|x|\geq R\}$ in $\bigcup_n [R+n2\pi N_k, R+(n+1)2\pi N_k] \cup [-R-(n+1)2\pi N_k, -R-n2\pi N_k]$ and we use the periodicity of $\psi_k(t)u_0 - L(t)u_0$ to get
$$
\| 1_{|x|\geq R} (\sqrt{1+x^2})^{-\alpha}D^{3/4}(\psi_k(t)u_0 - L(t)u_0)\|_{L^2} \lesssim \sum_n  (\sqrt{1+(Rn2\pi N_k)^2})^{-\alpha} p_{\pi N_k,3/4}(\psi_k(t)u_0 - L(t)u_0)
$$
and hence
$$
\| 1_{|x|\geq R} (\sqrt{1+x^2})^{-\alpha}D^{3/4}(\psi_k(t)u_0 - L(t)u_0)\|_{L^2} \lesssim (\sqrt{1+R^2})^{1-\alpha} N_k^{-1} \Lambda
$$
therefore, we get the result. \end{proof}

\begin{proof}[Proof of Proposition \ref{prop-luc}] Let $u_0 \in A$. First, since 
$$
\|\chi^{1/3} |L(\tau)u_0| \|_{L^6_\tau, L^6_x} \leq \|\chi^{1/3-1/\kappa}\|_{L^{6q/(q-6)}}^{1/6}\|\chi^{1/\kappa} |L(t)u|\|_{L^{q}_\tau,L^{q}_x}\leq C\Lambda \; ,
$$
and since $ \|\chi^{1/3-1/\kappa}\|_{L^{q/(q-6)}}$ is less than $\max_{\kappa \in [6,12]}M^{1/3-1/\kappa}$, with $M$ the maximum of $\|\chi\|_{L^1}$ and $\|\chi\|_{L^\infty}$ (making the constant independent of $q$ and $\kappa$), the flow $\psi$ is well-defined for the considered times and 
$$
\|\psi(t)u_0 - L(t)u_0\|_{L^\infty([-T,T],H^1)} \leq \Lambda\; , \|p_{\pi N_k, 1}(\psi_k(t)u_0 - L(t)u_0)\|_{L^\infty[-T,T]} \leq \Lambda\; .
$$
Let $w= \psi(t)u_0 - \psi_k(t)u_0$. This function satisfies
$$
i\partial_t w - \sqrt{1 - \lap} w  = (1- \lap)^{-1} \Big( \chi \re(\psi(t)u_0)^3- \Pi_k P_k \chi \re (\psi_k(t) u_0)^3 \Big) \; .
$$
We can rewrite this equation as 
$$
w(t) = -i \int_{0}^t L(t-\tau) (1-\lap)^{-1/2} \Big( \chi  \re(\psi(\tau)u_0)^3 -  \Pi_k P_k \chi \re (\psi_k(\tau) u_0)^3 \Big)d\tau
$$
or 
$$
w(t) = -i\int_{0}^t L(t-\tau) (1- \lap)^{-1/2} (I + II +III )d\tau
$$
with
\begin{eqnarray*} 
I & = & (1- \Pi_k) \chi  \re(\psi(\tau)u_0)^3 \\
II &= & \Pi_k (1-P_k) \chi  \re(\psi(\tau)u_0)^3 \\
III & = & \Pi_k P_k \chi \Big(   \re(\psi(\tau)u_0)^3 - \re(\psi_k (\tau)u_0)^3\Big)
\end{eqnarray*}

We apply the previous lemma with $R= r_k = \pi N_k (1-\frac1k)$. We have that
$$
\|  (\sqrt{1+x^2})^{-\alpha}D^{3/4}(\psi(t)u_0 - \psi_k(t)u_0)\|_{L^2(\R)} \leq C(r_k) \Lambda + p_{r_k,3/4} (w)
$$
for the considered times.

Let us estimate the second part of the right hand side. We have 
$$
p_{r_k, 3/4} (w) \leq \int_{0}^t \Big(  p_{r_k + 1, -1/4} (I)  +  p_{r_k + 1, -1/4} (II) +  p_{r_k+ 1, -1/4} (III) \Big) d\tau\; .
$$
For $I$, we use that
$$
p_{r_k+ 1, -1/4} (I)\leq \|I\|_{H^{-1/4}} \leq R_k^{-1/4} \|\chi  \re(\psi(\tau)u_0)^3\|_{L^2}
$$
hence 
$$
\int_{0}^t \Big(  p_{r_k+ 1, -1/4} (I)d\tau  \lesssim T^{1/2} \Lambda^3 R_k^{-1/4} \; .
$$
Indeed, we have 
$$
\|\chi  \re(\psi(\tau)u_0)^3\|_{L^2} = \|\chi^{1/3}  \re(\psi(\tau)u_0)\|_{L^6}^3 
$$
and as $\psi(t) u_0 = L(t) u_0 + (\psi(t) u_0 - L(t)u_0)$ , we get
$$
\|\chi  \re(\psi(\tau)u_0)^3\|_{L^2} \leq  \|\chi^{1/3}  L(\tau)u_0\|_{L^{6}}^3 + \|\chi^{1/3}  (\psi(t)u_0 - L(\tau)u_0)\|_{L^{6}}^3\leq C \Lambda \; ,
$$
with a constant $C$ that depends only on $\chi$ (and not on $q$ or $\kappa$).
The $L^\infty_t , H^1_x$ norm of $\psi(t) u_0 - L(t) u_0$ is bounded by $\Lambda$ thanks to Proposition \ref{prop-lwp}. Therefore, we get
$$
\int_{0}^t \Big(  p_{r_k+ 1, -1/4} (I)d\tau  \lesssim R_k^{-1/4}\Big( T^{1-3/q} \|\chi^{1/\kappa}  L(\tau)u_0\|_{L^q_\tau, L^q_x}^3 + T \Lambda^3 \Big) \leq R_k^{-1/4} T^{1/2} \Lambda^3 \; .
$$

For $II$, we use that $r_k + 1 = \pi N_k (1 - \frac1k + \frac1{\pi N_k})$ that $\pi N_k = \pi 2^k > k$ and that for $x\in [-\pi N_k, \pi N_k]$, $P_k f (x) = f(x)$ to get
$$
  p_{r_k + 1, -1/4} (II) =0 \; .
$$
%and then a H\"older inequality with $ \frac12  =  \frac{q-6}{2q} + \frac3q$ to get
%\begin{eqnarray*}
%p_{R+ 1, -1/4} (II) & \leq & \|\chi^{1/3-1/\kappa}\|_{L^{6q/(q-6)}(]-\infty,-\pi N_k]\cup[\pi N_k,\infty[)}^{3}\|\chi^{3/\kappa }|\psi(t)u_0|^3\|_{L^{q/3}} \\
%& = & \|\chi^{1/3-1/\kappa}\|_{L^{6q/(q-6)}(]-\infty,-\pi N_k]\cup[\pi N_k,\infty[)}^{3} \|\chi^{1/6 }|\psi(t)u_0|\|_{L^{q}}^{3}\; .
%\end{eqnarray*}
% Finally, we have 
%$$
%\int_{0}^t \Big(  p_{R+ 1, -1/4} (II)d\tau \lesssim  \|\chi^{1/3-1/\kappa}\|_{L^{6q/(q-6)}(]-\infty,-\pi N_k]\cup[\pi N_k,\infty[)}^{3}\Lambda^3 \; .
%$$
%We note that $1/3-1/\kappa$ is bigger than $\frac16$ and that $\frac{6q}{q-6}\geq 6$ hence the norm
%$$
% \|\chi^{1/3-1/\kappa}\|_{L^{6q/(q-6)}(]-\infty,-\pi N_k]\cup[\pi N_k,\infty[)}
%$$
%is well defined and goes to $0$ when $k$ goes to $\infty$.

For $III$, we use that $III$ is $2\pi N_k$ periodic, thus its $p_{r_k+1, -1/4}$ norm is less than its $p_{\pi N_k,-1/4}$ norm, and then, it is less than
$$
\int_{0}^t p_{r_k+ 1, -1/4} (III) d\tau \leq C T^{1/2} \Lambda^{2} \|\chi^{1/3} w\|_{L^\infty([-\pi N_k, \pi \N_k]}) \; .
$$
We recall from the assumptions on $\chi$ that $0 \leq \chi \leq C  (\sqrt{1+x^2})^{-3 \alpha}$. Since $3/4 > 1/2$, we have the Sobolev embedding $L^\infty \subseteq H^{3/4}$, hence
$$
\int_{0}^t p_{\pi N_k + 1, -\mu} (III) d\tau \leq C T^{1/2} \Lambda^{2} \| (\sqrt{1+x^2})^{-\alpha}D^{3/4}w\|_{L^\infty([-T,T],L^2)} \; .
$$

Finally, for $T = \frac1{C\Lambda^{4}}$ with $C$ big enough, we have
$$
\| (\sqrt{1+x^2})^{-\alpha}D^{3/4}w\|_{L^\infty([-T,T],L^2)} \lesssim \Lambda^3 R_k^{-1/4} + C(r_k)
$$
and since $r_k$ goes to $\infty$ and $R_k^{-1/4}$ go to $0$ when $k \rightarrow \infty$, there exists $k_0$ (depending on $\Lambda$, $q$) such that
$$
\| (\sqrt{1+x^2})^{-\alpha}D^{3/4}w\|_{L^\infty([-T,T],L^2)}\leq \varepsilon \; .
$$
%For $R \leq \pi N_k$, we use that
%$$
%\|p_{R,3/4}(w)\|_{L^\infty([-T,T])}\leq \|p_{\pi N_k,3/4}(w)\|_{L^\infty([-T,T])}\leq \varepsilon\; .
%$$
%Hence, for all $R $ and all $t \in [-T,T]$, there exists $k_0$ such that for all $k\geq k_0$,
%$$
%p_{R,3/4}(w)\leq \varepsilon \; .
%$$
\end{proof}

\begin{corollary} We fix $\Lambda \geq 1$. For all time $t$ such that $|t|\leq \frac1{C\Lambda^4}$, the sequence $\psi_k(t)$ converges uniformly in $A$ in $\Hloc$. \end{corollary}

\begin{proof} The $\| (\sqrt{1+x^2})^{-\alpha}D^{3/4} \cdot \|_{L^2}$ norm controls the $p_{R,s}$ semi-norm for all $s<1/2$, hence the uniform convergence for all the $p_{R,3/4}$ implies the uniform convergence in $\Hloc$. Besides, $L(t)$ is uniformly continuous for the topology of $\Hloc$. \end{proof}

\subsection{Local continuity of the flows}

In this subsection, we prove the continuity of the flows $\psi$ and $\psi_k$ with respect to the initial datum in the support of $\mu$.

\begin{lemma}\label{lem-PQ} Let $q> p \geq 2$. Let $\kappa_1 > \kappa_2 \geq 6$ such that $p\geq \kappa_1$, $q\geq \kappa_2$ and $1/\kappa_2 - 1/\kappa_1 \geq 1/p - 1/q$. The map
$$
u_0 \mapsto \chi^{1/\kappa_2} L(\tau) u_0
$$
is uniformly continuous from $\Hloc$ to $L^p_\tau, L^p_x$ on the sets $\{u\in \Hloc \; |\; \|\chi^{1/\kappa_1}L(\tau) u\|_{L^q_\tau,L^q_x} \leq \Lambda\}$.
\end{lemma}

\begin{proof}Let $u_1$ and $u_2$ such that 
$$
 \|\chi^{1/\kappa_1}L(\tau) u_1\|_{L^q_\tau, L^q_x} \leq \Lambda \mbox{ and }  \|\chi^{1/\kappa_1}L(\tau) u_2\|_{L^q_\tau,L^q_x} \leq \Lambda \; .
$$
We recall that 
$$
d(u_1,u_2) = \sum_{l,k} 2^{-(l+k)}\frac{p_{k, 1/2-1/l}(u_1-u_2)}{1+p_{k, 1/2-1/l}(u_1-u_2) } \; .
$$
We compare $\chi^{1/\kappa_2} L(\tau) u_1$ and $\chi^{1/\kappa_2} L(\tau) u_2$, we have 
$$
\|\chi^{1/\kappa_2} L(\tau) u_1 - \chi^{1/\kappa_2} L(\tau) u_2\|_{L^p_x}^p = \int_{-R}^R \chi^{p/\kappa_2} |L(\tau) (u_1-u_2)|^p + \int_{|x|\geq R} \chi^{p/\kappa_2} |L(\tau) (u_1 - u_2)|^p dx \; .
$$
Since $\chi$ is bounded and $L^p$ is embedded in $H^{s}$ with $s=1/2 - 1/p$, we get thanks to H\"older inequalities with $1/p = 1/q + 1/r$ and $\alpha = 1/\kappa_2 - 1/\kappa_1$ such that $\alpha r \geq 1$ : 
\begin{multline*}
\|\chi^{1/\kappa_2} L(\tau) u_1 - \chi^{1/\kappa_2} L(\tau) u_2\|_{L^p_x} \leq  \|\chi^{1/\kappa_2}\|_{L^\infty([-R,R])} \|L(\tau)(u_1-u_2)\|_{L^p([-R,R])} + \\
\|\chi^{\alpha}\|_{L^r(\{|x|>R\}} \|\chi^{1/\kappa_1} L(\tau) (u_1-u_2)\|_{L^q}\\
\lesssim  \|\chi\|_{L^\infty}^{1/\kappa_2} p_{R',s}(L(\tau) (u_1-u_2) + \left( \int_{|x|\geq R} \chi^{r\alpha} \right)^{1/r} \left( \int \chi^{q/\kappa_1} |L(\tau)(u_1-u_2)|^q\right)^{1/q} \; ,
\end{multline*}
where $R' > R$ to compensate the boundary effects.

We use the definition of $d$ to get
\begin{eqnarray*}
\|\chi^{1/\kappa_2} L(\tau) u_1 - \chi^{1/\kappa_2} L(\tau) u_2\|_{L^3_x}  &\lesssim & \|\chi\|_{L^\infty}^{1/\kappa_2} 2^{R'+|\tau|+s}d(u_1,u_2) + \\
& & \left( \int_{|x|\geq R} \chi^{\alpha} \right)^{1/r}\Big( \|\chi^{1/\kappa_1} L(\tau) u_1 \|_{L^q} + \|\chi^{1/\kappa_1} L(\tau) u_2 \|_{L^q} \Big) \; .
\end{eqnarray*}
We integrate over $\tau \in [-1,1]$. We get
$$ 
\|\chi^{1/\kappa_2} L(\tau) u_1 - \chi^{1/\kappa_2} L(\tau) u_2\|_{L^p_\tau,L^p_x} \lesssim \|\chi\|_{L^\infty}^{1/\kappa_2} 2^{R'+1+s} d(u_1,u_2) + \left( \int_{|x|\geq R} \chi^{r\alpha}\right)^{1/r} \Lambda \; .
$$
Let $\varepsilon > 0$, we take $R$ big enough such that $ \left( \int_{|x|\geq R} \chi^{r\alpha}\right)^{1/r} \leq \varepsilon$. For all $u_1$ and $u_2$ such that $d(u_1,u_2) \leq 2^{-(1+R'+s)}\varepsilon $, we have 
$$ 
\|\chi^{1/\kappa_2} L(\tau) u_1 - \chi^{1/\kappa_2} L(\tau) u_2\|_{L^p_\tau,L^p_x} \lesssim \varepsilon
$$
which concludes the proof.\end{proof}

\begin{proposition}\label{prop-contpsi} Let $q\geq 9$ and $\kappa \in [6,12]$ such that $1/6\geq 1/\kappa -1/q$ and $q\geq \kappa$. The flow $\psi(t)$ is uniformly continuous from $\Hloc $ to $\Hloc$ in 
$$
\{u\in \Hloc \; |\; \|\chi^{1/\kappa}L(\tau) u\|_{L^q_\tau,L^q_x} \leq \Lambda\}
$$ 
for all time $|t| \leq T = \frac1{C\Lambda^4}$ with $C$ a constant independent from $q$ and $\kappa$. \end{proposition}

\begin{proof} Let $u_1$ and $u_2$ in $\{u\in \Hloc \; |\; \|\chi^{1/\kappa}L(\tau) u\|_{L^q_\tau,L^q_x} \leq \Lambda\}$. First, in this set, the flow $\psi(t)$ is locally well-posed. Indeed, we have Proposition \ref{prop-lwp} and
\begin{eqnarray*}
\|\chi^{1/3}L(\tau) u_0\|_{L^6_\tau,L^6_x} & \leq & \|\chi^{1/3 -1/\kappa} \|_{L^{6q/(q-6)}} \|\chi^{1/\kappa}L(\tau)u_0\|_{L^q_\tau,L^q_x} \\
& \leq & \max(\|\chi\|_{L^1},\|\chi\|_{L^\infty})^{1/3-1/\kappa} \Lambda \\
& \leq & \max_{\kappa \in [6,12]} \max(\|\chi\|_{L^1},\|\chi\|_{L^\infty})^{1/3-1/\kappa} \Lambda \; .
\end{eqnarray*}
The solutions can be written
$$
\psi(t) u_i = L(t) u_i + v_i
$$
with $v_i \in \mathcal C ([-T,T],H^1)$ and $\|v_i\|_{L^\infty_t,H^1_x} \lesssim \Lambda$. We compare $v_1$ and $v_2$ in $H^{1}$. We have
$$
v_1(t) - v_2(t) = \int_{0}^t L(t-\tau) (1-\lap)^{-1/2} \chi  \Big( \re(L(\tau)u_1 + v_1(\tau))^3 -\re(L(\tau)u_2 + v_2(\tau))^3 \Big)\; .
$$
Taking its $H^{1}$ norm yields
$$
\|v_1(t) - v_2(t)\|_{H^{1}} \leq \int_{0}^t \| \chi \Big( \re^3(L(\tau)u_1 + v_1(\tau)) -\re^3(L(\tau)u_2 + v_2(\tau)) \Big)\|_{L^2} \; .
$$
We use that
\begin{eqnarray*}
\Big|  \re^3(L(\tau)u_1 + v_1(\tau)) -\re^3(L(\tau)u_2 + v_2(\tau)) \Big|& \lesssim & (|v_1-v_2| + |L(\tau) u_1 - L(\tau) u_2|) \times \\
 & & (|v_1|^2+|v_2|^2+ |L(\tau)u_1|^2+|L(\tau)u_2|^2)
\end{eqnarray*}
to bound
$$
\| \chi \Big( \re(L(\tau)u_1 + v_1(\tau))^3 -\re(L(\tau)u_2 + v_2(\tau))^3 \Big)\|_{L^2} 
$$
by
\begin{multline*}
C(\|\chi^{1/3}( v_1-v_2)\|_{L^6} + \|\chi^{1/3}( L(\tau) u_1 - L(\tau) u_2)\|_{L^6})\times \\
 (\|\chi^{1/3} v_1\|_{L^6}^2+ \|\chi^{1/3} v_2\|_{L^6}^2+ \|\chi^{1/3} L(\tau)u_1\|_{L^6}^2+\|\chi^{1/3} L(\tau)u_2\|_{L^6}^2)\; .
\end{multline*}
By integrating it over time, we get that
$$
\int_{0}^t \| \chi \Big( \re(L(\tau)u_1 + v_1(\tau))^3 -\re(L(\tau)u_2 + v_2(\tau))^3 \Big)\|_{L^2} d\tau 
$$
is bounded by
\begin{multline*}
C\Big( \left(\int_{0}^t \|\chi^{1/3}( v_1-v_2)\|_{L^6}^3d\tau\right)^{1/3} +  \|\chi^{1/3} (L(\tau) u_1 - L(\tau) u_2)\|_{L^6_\tau,L^6_x}\Big)\times \\
\Big( \left( \|\chi^{1/3} v_1\|_{L^6_x}^3d\tau \right)^{2/3}+ \left( \int_{0}^t \|\chi^{1/3} v_2\|_{L^6}^3\right)^{2/3}+ |t|^{1/3}\|\chi^{1/3} L(\tau)u_1\|_{L^6_\tau,L^6_x}^2+|t|^{1/3}\|\chi^{1/3} L(\tau)u_2\|_{L^6_\tau,L^6_x}^2)\; .
\end{multline*}
We first notice that the $L^6_x$ norm of $\chi^{1/3} L(\tau) u_i$ is less than $\max_\kappa \max(\|\chi\|_{L^1},\|\chi\|_{L^\infty})^{1/3-1/\kappa}$ times the $L^q_x$ norm of $\chi^{1/\kappa} L(\tau) u_i$ thanks to H\"older inequality. Then, the $L^6_x$ norm of $\chi^{1/3} v_i$ is less than $\|\chi\|_{L^2_x}^{1/3}$ times the $L^\infty$ and hence the $H^1$ norm of $v_i$. We get
$$
\int_{0}^t \|\chi^{1/3} v_i\|_{L^3_x}^3d\tau  \lesssim |t| \Lambda^3 \mbox{ and } \|\chi^{1/3} L(\tau)u_i\|_{L^6_\tau,L^6_x}\lesssim \Lambda \; .
$$
Finally, we have, as $|t|^{2/3} \leq |t|^{1/3}$,
$$
\|v_1(t)- v_2(t)\|_{H^1} \lesssim (|t|^{1/3} \|v_1-v_2\|_{L^\infty_T, H^{1}_x} +  \|\chi^{1/3}( L(\tau) u_1 - L(\tau) u_2)\|_{L^3_\tau,L^6_x}) |t|^{1/3}\Lambda^2 \; .
$$
Hence, for $T \leq \frac1{C\Lambda^{4}}\leq \frac1{C\Lambda^{3}}$ with $C$ big enough, we get
$$
\|v_1-v_2\|_{L^\infty_T,H^{1}} \lesssim   \|\chi^{1/3} (L(\tau) u_1 - L(\tau) u_2)\|_{L^6_\tau,L^6_x}\; .
$$
Thanks to the previous uniform continuity lemma (Lemma \ref{lem-PQ}) with $p=6$, $\kappa_1= \kappa$, $\kappa_2 = 3$, we have that $u \mapsto \psi(t)u- L(t)u$ is continuous from $\Hloc$ to $H^{1}$ and thus from $\Hloc$ to $\Hloc$. We have proved in the previous section that $L(t)$ was uniformly continuous from $\Hloc$ to $\Hloc$ hence so is $\psi(t)$.\end{proof}

 \begin{proposition}\label{prop-contpsik} Let $q\geq 9$ and $\kappa \in [6,12]$. The flow $\psi_k(t)$ is uniformly continuous from $\Hloc $ to $\Hloc$ in 
 $$
 \{u\in \Hloc \; |\; \|\chi^{1/\kappa}L(\tau) u\|_{L^q_\tau,L^q_x} \leq \Lambda\}
 $$
 for all time $|t| \leq T = \frac1{C\Lambda^4}$ with a constant $C$ independent from $k$, $\kappa$ and $q$. \end{proposition}

\begin{proof} We proceed in the same way, proving first that the map
$$
u_0 \mapsto \Pi_k P_k \chi^{1/3} L(\tau) u_0
$$
is uniformly continuous from $\Hloc$ to $L^6_\tau, L^6_x([-\pi N_k, \pi N_k])$ on the sets 
$$
\{u\in \Hloc \; |\; \|\chi^{1/\kappa}L(\tau) u\|_{L^q_\tau,L^q_x} \leq \Lambda\}.
$$
Then, we prove that $u\mapsto \psi_k(t) u -L(t) u$ is uniformly continuous from $\Hloc$ to $p_{\pi N_k,1/2}$, keeping in mind that $\psi(t) u -L(t) u$ belongs to $E_k$ which ensures the uniform continuity of $\psi_k(t)$.\end{proof}

\begin{proposition}\label{prop-gcpk} For all time $t \in \R$, the flow $\psi_k(t)$ is continuous from $\Hloc$ to $\Hloc$ on the support of $\rho$. \end{proposition}

\begin{proof} Let $u_1$ such that $\|\chi^{1/6}L(\tau) u_1\|_{L^{12}_\tau([-t,t], L^{12}([-\pi N_k,\pi N_k]))} = \Lambda (t)$ and let $u_2 \in \Hloc$. Note that $\Lambda(t)$ increases with $t$. We write $\psi_k(t) u_i = L(t) u_i + v_i(t)$. In the proof of Proposition \ref{prop-gwp}, we proved that
$$
p_{\pi N_k,1} (v_i (t)) \leq \mH(v_i) \leq C_k  \int_{0}^t \|L(\tau) u_i\|_{L^6_{[-\pi N_k, \pi N_k]}}^3 d\tau \; e^{c_k\int_{0}^t  \|L(\tau) u_i\|_{L^6_{[-\pi N_k, \pi N_k]}} d\tau} \; .
$$
Hence, for $v_1$, we get
$$
p_{\pi N_k,1} (v_1(t)) \leq C_k |t|^{3/4} (\Lambda (t) )^3 e^{c_k |t|^{11/12} \Lambda(t)} \; .
$$
For $v_2$, we have to consider that
$$
\|L(\tau) (u_1 - u_2)\|_{L^6([-\pi N_k, \pi N_k])} \leq p_{\pi N_k, 1/3} (L(\tau) (u_1 - u_2) ) \leq p_{\pi N_k +|\tau| , 1/3} (u_1 - u_2) \leq C_k(\tau) d(u_1,u_2) \; .
$$
Note that $C_k(\tau)$ can be chosen to be increasing with $|\tau|$. Therefore, we get
$$
\| L(\tau) (u_2) \|_{L^6([-t,t],L^6([-\pi N_k, \pi N_k])} \leq |t|^{1/4}\Lambda + C'_k(t) d(u_1,u_2)
$$
where 
$$
C_k'(t) = \Big(\int_{-t}^t C_k(\tau)^6 d\tau \Big)^{1/6}
$$
is increasing with $t$. We get
$$
p_{\pi N_k,1} (v_2(t)) \leq C_k |t|^{3/4} (|t|^{1/4}\Lambda (t) + C_k' (t)d(u_1,u_2) )^3 e^{c_k |t|^{11/12} (\Lambda(t) + C_k'(t)d(u_1,u_2))} \; .
$$
We bound $\psi_k(t) u_i$. We have 
$$
\|L(\tau) L(t) u_1\|_{L^{12}_\tau, L^{12}([-\pi N_k,\pi N_k])} \leq \Lambda(t+1) 
$$
 and 
 $$
 \|L(\tau) L(t) u_2\|_{L^{12}_\tau, L^{12}([-\pi N_k,\pi N_k])} \leq \Lambda(t+1) + C_k'(t+1) d(u_1,u_2)\; .
$$
We also have 
\begin{eqnarray*}
\|L(\tau) v_i \|_{L^{12}_\tau, L^{12}([-\pi N_k, \pi N_k])} &\leq & C_k^{(2)}(t) p_{\pi N_k, 1}(v_i) \\
&\leq & C_k^{(2)}(t) C_k |t|^{3/4} (|t|^{1/4}\Lambda (t) + C_k' (t)d(u_1,u_2) )^3 e^{c_k |t|^{11/12} (\Lambda(t) + |t|^{5/6}C_k'(t)d(u_1,u_2))}
\end{eqnarray*}
with $C_k^{(2)}(t)$ a constant increasing with $t$. Hence, for all $u_2$ such that $d(u_1,u_2) \leq \eta $ with $\eta > 0$, we have 
$$
\|L(\tau) \psi(t) u_2\|_{L^{12}_\tau, L^{12}([-\pi N_k, \pi N_k])} \leq D_k(t) 
$$
with
$$
D_k(t)= \Lambda(t+1) + C_k'(t+1) \eta +
C_k^{(2)}(t) C_k |t|^{3/4} (|t|^{1/4}\Lambda (t) + C_k' (t)\eta )^3 e^{c_k |t|^{11/12} (\Lambda(t) + |t|^{5/6}C_k'(t)\eta)}\; .
$$
Suppose that $\psi(t)$ is not continuous in $u_1$ for all $t \geq 0$. Let 
$$
t_0 = \inf \{ t\geq 0 \; |\; \psi(t) \textrm{ is not continuous in } u_1 \}\; .
$$
First case, if $\psi_k(t_0) $ is continuous then since $\|\psi_k(t_0) u_i\|_{L^6_\tau,L^6([-\pi N_k, \pi N_k])} \leq D_k(t_0)$, we get that $\psi_k(t)  = \psi_k(t - t_0)\circ \psi_k(t_0)$ is continuous for all $t \in [t_0 - T, t_0 + T]$ with $T = \frac1{C D_k(t)^4}$ thanks to Proposition \ref{prop-contpsik}, which is absurd.

In the second case, if $\psi_k(t_0)$ is not continuous, we choose $\varepsilon \leq \frac1{2CD_k(t_0)^4} \leq \frac1{C D_k(t_0-\varepsilon)^4}$. As $\psi_k(t_0 - \varepsilon)$ is continuous, so is $\psi_k(t_0) = \psi_k(\varepsilon) \circ \psi_k ( t_0 - \varepsilon)$, which is absurd. Therefore, $\psi_k(t)$ is continuous for all $t \geq 0$. The proof is analogous for $t\leq 0$. \end{proof}

\section{Global theory}

\subsection{Definition of the set where there is global well posedness}

In this section, we prove that on some sets, the sequence of flows $\psi_k(t)$ converge uniformly towards $\psi(t)$ for all time.

Let us first describe the set where we have these properties.

\begin{definition} Let $\Lambda \geq 1$. For all $n\geq 1$, we call $\Lambda_n = n^{1/8} \Lambda$ and $T_n = \frac1{C\Lambda_n^4} = \frac1{\sqrt n}\frac1{C \Lambda^4}$ and for all $n\in \N$, $t_n = \sum_{j=1}^n T_j$ and $q_n = 12 - \sum_{k=0}^n (\frac23)^k$. For all $n$ and all $k$, let
$$
A_{k,n}(\Lambda) = \lbrace u \in \Hloc \; \Big| \; \| \chi^{1/q_n}  L(\tau)  \psi_k(t_n) u \|_{L^{q_n}_\tau,L^{q_n}_x} \leq \Lambda_{n+1} \rbrace \; ,
$$
and 
$$
A_k(\Lambda) = \bigcap_{n\in \N} A_{k,n}(\Lambda) \mbox{ and } A(\Lambda) = \limsup_{k\rightarrow \infty} A_k(\Lambda) \; .
$$
Finally we call $A$ the union of the $A(\Lambda)$ for all $\Lambda \geq 1$ :
$$
A = \bigcup_{\Lambda\geq 1} A(\Lambda) \; .
$$
\end{definition}

\subsection{Global well-posedness and global uniform convergence}

\begin{proposition} \label{prop-guc} For all $t \in \R$ and all $u \in A(\Lambda)$, $\psi(t) u$ is defined and unique in $L(t) u  + L^\infty_{loc,\tau}, H^1$. Besides, $\psi_k(t)$ converges uniformly on $A(\Lambda)$ towards $\psi(t)$ in $\Hloc$. Finally, $\psi(t)$ is continuous in $A$ for the topology of $\Hloc$\end{proposition}

Before proving this proposition, we need to prove the following lemma.

\begin{lemma} There exists $C$ such that for all $n\in\N$ and $f \in H^1(\R)$ or $f \in H^1([-\pi N_k, \pi N_k])$, we have 
$$
\|\chi^{1/q_n}L(\tau) f\|_{L^{q_n}_\tau,L^{q_n}_x } \leq C  \|f\|_{H^1}\; .
$$
\end{lemma} 

\begin{proof} We have 
$$
\|\chi^{1/q_n}L(\tau) f\|_{L^{q_n}_\tau,L^{q_n}_x } \leq  \|\chi\|_{L^{1}}^{1/q_n} \|L(\tau) f\|_{L^{q_n}_\tau,L^\infty}
$$
then, we use Sobolev embedding $L^\infty \subset H^1$ to
$$
\|L(\tau)f\|_{L^\infty} \leq \|L(\tau) f\|_{H^1} = \|f\|_{H^1}
$$
in the case of $\R$ and 
$$
\|L(\tau)f\|_{L^\infty} \leq \|L(\tau) f\|_{H^1} = p_{\pi N_k+1, 1}(f) \leq 2 p_{\pi N_k,1}(f) 
$$
in the periodic case. 

Therefore,
$$
\|\chi^{1/q_n}L(\tau) f\|_{L^{q_n}_\tau,L^{q_n}_x } \leq 2 \max_{q\in [9,12]}(\|\chi\|_{L^1}^{1/q_n}) \|f\|_{H^1}\; .
$$
\end{proof}

\begin{proof}[Proposition \ref{prop-guc}.] We assume $t\geq 0$. The argument for $t\leq 0$ is similar. We prove by induction on $n$ the property $P(n)$ : for all $t \in [t_n, t_{n+1}]$, the flow $\psi(t) $ is well defined on $A(\Lambda)$ and unique in $L(t-t_n) \psi(t_n)u + L^\infty([t_n,t_{n+1}], H^1)$,  $\psi_k(t)$ converges uniformly in $A(\Lambda)$ towards $\psi(t)$, and there exists $k_0$ such that for $k\geq k_0$, the $L^{q_n}_\tau, L^{q_n}_x$ norm of $\chi^{1/q_n} L(\tau) \psi(t_n)u$ is bounded by $2\Lambda_{n+1}$.

For $n=0$, as $t_0= 0$, by definition of $A(\Lambda)$ if $u\in A(\Lambda)$, then there exists a sequence $k_j \rightarrow \infty$ such that $ u \in A_{k_j}(\Lambda)$ for all $j\in \N$. In particular, $u\in A_{k_0,0}(\Lambda)$ which means that 
$$
 \| \chi^{1/q_0}  L(\tau)   u \|_{L^{q_0}_\tau,L^{q_0}_x} \leq \Lambda_{1}
$$
and thanks to the Proposition \ref{prop-lwp} $\psi(t) u$ is well-defined in for times $[0,T_1] = [t_0,t_1]$ and $\psi_k(t)$ converges uniformly (Proposition \ref{prop-luc}) towards $\psi(t)$ in this time interval.

Let us show that $P(n-1)$ induces $P(n)$. Let us first prove the bound on  $\chi^{1/q_n} L(\tau) \psi(t_n)u$. We have 
\begin{multline*}
\psi_k(t_n) u = \Big(\psi_k(T_n) (\psi_k(t_{n-1}) u) -L(T_n) \psi_k (t_{n-1})u \Big) + \\
\Big( L(T_n)\psi_k(t_{n-1})u - \psi(T_n) \psi_k(t_{n-1} )u\Big) + \psi(T_n) \psi_k(t_{n-1}) u\; .
\end{multline*}
We have 
\begin{multline*}
 \|\chi^{1/q_n} L(\tau) \Big(\psi_k(T_n) (\psi_k(t_{n-1} )u) -L(T_n) \psi_k (t_{n-1})u \Big)\|_{L^{q_n}_\tau,L^{q_n}_x} \leq \\
 C \|\chi\|_{L^{1}}^{1/q_n} p_{\pi N_k, 1} (\psi_k(T_n) (\psi_k(t_{n-1} )u -L(T_n) \psi_k (t_{n-1})u )\; .
\end{multline*}
In the local theory, we have seen that for all $\varepsilon$, we can take a constant $C$ big enough in the definition of $T_n$ such that
$$
p_{\pi N_k, 1} (\psi_k(T_n) (\psi_k(t_{n-1} )u -L(T_n) \psi_k (t_{n-1})u ) \leq 2 \varepsilon \Lambda_{n} \max(\|\chi\|_{L^1},\|\chi\|_{L^\infty})^{1/3-1/q_n}
$$
since 
$$\|\chi^{1/3} L(\tau) \psi_k(t_{n-1}) u\|_{L^6_\tau,L^6_x} \leq \max_{q\in [9,12]} \max(\|\chi\|_{L^1},\|\chi\|_{L^\infty})^{1/3-1/q} 2 \Lambda_n
$$ 
above a certain $k_0$. We take $\varepsilon$ such that $2 \varepsilon \Lambda_n  \max_{q\in [9,12]} \max(\|\chi\|_{L^1},\|\chi\|_{L^\infty})^{1/3-1/q}C $ is less than $\frac14$. For the same reason, we can bound 
$$
 \|\chi^{1/q_n} L(\tau) \Big(\psi(T_n) (\psi_k(t_{n-1} u) -L(T_n) \psi_k (t_{n-1})u \Big)\|_{L^{q_n}_\tau,L^{q_n}_x}
$$
by $ \frac14 \Lambda_n$ for $k\geq k_0$. Since for $k\geq k_0$, $\|\chi^{1/q_{n-1}} L(\tau) \psi_k(t_{n-1}) u\|_{L^{q_{n-1}}_\tau,L^{q_{n-1}}_x}$ is bounded by $2\Lambda_n$, we have that $\|\chi^{1/q_{n-1}} L(\tau) \psi(t) \psi_k(t_{n-1}) u\|_{L^{q_{n-1}}_\tau,L^{q_{n-1}}_x}$ is bounded by $C\Lambda_n$, thus $u\mapsto \chi^{1/q_n} L(\tau)u$ is continuous from $\Hloc$ to $L^{q_n}_\tau, L^{q_n}_x$ in the set of functions we consider, as $q_{n-1} > q_n$. Therefore, since $\psi(T_n) \psi_k(t_{n-1}) u$ converges uniformly in $u$ on the set we consider, there exists $k_1$ such that for all $k\geq k_1$
$$
\|\chi^{1/q_n} L(\tau) \psi(T_n)\psi_k(t_{n-1}) u\|_{L^{q_n}_\tau,L^{q_n}_x} \leq  \frac32 \|\chi^{1/q_n} L(\tau) \psi(t_n) u\|_{L^{q_n}_\tau,L^{q_n}_x} \leq \frac32 \Lambda_{n+1}\; .
$$
Finally, for al $k\geq k_1$,
$$
\|\chi^{1/q_n} L(\tau) \psi(t_n)u\|_{L^{q_n}_\tau, L^{q_n}_x} \leq 2 \Lambda_{n+1}\; ,
$$
which proves the bound on $\chi^{1/q_n} L(\tau) \psi (t_n) u$.

Let $t \in [t_n, t_{n+1}]$. We write $t = t_n +t'$, with $t'\in [0,T_{n+1}]$. As for all $j$
$$
 \| \chi^{1/3}  L(\tau)  \psi_{k_j}(t_n) u \|_{L^6_\tau,L^6_x} \lesssim \Lambda_{n+1}\;  ,
$$
and $\psi_k(t_n)$ converges uniformly towards $\psi(t_n)$ we have 
$$ 
\| \chi^{1/3}  L(\tau)  \psi(t_n) u \|_{L^6_\tau,L^6_x} \lesssim \Lambda_{n+1} 
$$
and thus $\psi(t')\psi(t_n)u = \psi(t) u$ is defined and unique thanks to the local theory.

We have the bound, for $k\geq k_1$,
$$
 \| \chi^{1/q_n}  L(\tau)  \psi_k(t_n) u \|_{L^{q_n}_\tau,L^{q_n}_x} \leq 2 \Lambda_{n+1} \; .
$$

We compare $\psi(t) u$ and $\psi_k(t) u$. We have 
$$
\psi(t) u - \psi_k(t) u = \psi(t') \psi(t_n) u - \psi(t') \psi_k(t_n) u + \psi(t')\psi_k(t_n)u - \psi_k(t')\psi_k(t_n) u \; .
$$
Since $\psi(t')$ is uniformly continuous on the sets we consider (Proposition \ref{prop-contpsi}) and $\psi_k(t_n)$ uniformly converges towards $\psi(t_n)$ ($P(n-1)$) we get that $\psi(t') \circ \psi_k(t_n)$ converges uniformly towards $\psi(t)$. And because of the local uniform convergence of $\psi_k$ towards $\psi$ on the set where the $\psi_k$ belong (Proposition \ref{prop-luc}), we get that $\psi(t') \circ \psi_k(t_n) -\psi_k(t)$ converges uniformly towards $0$. Hence the result. These properties result from the fact that $q_n \geq 9$.

Finally, let $u\in A$, $t\in \R$, and $\varepsilon > 0$. For all $v \in A$, there exists $\Lambda_1$ and $\Lambda_2$ such that 
$$
u\in A(\Lambda_1) \mbox{ and } v\in A(\Lambda_2)\; .
$$
Hence, there exist $k_1$ and $k_2$ such that for all $k\geq k_1$,
$$
d(\psi(t) u , \psi_k(t) u) < \frac\varepsilon3
$$
and for all $k\geq k_2$,
$$
d(\psi(t) v,\psi_k(t) v) < \frac\varepsilon3 \; .
$$
Let $k = \max (k_1,k_2)$. As $\psi_k(t)$ is continuous in $\Hloc$, there exists $\eta > 0$, such that $d(u,v) \leq \eta$ implies
$$
d(\psi_k(t) u, \psi_k(t) v) < \frac\varepsilon3
$$
and thus
$$
d(\psi(t)u,\psi(t) v) < \varepsilon \; ,
$$
which concludes the proof of the continuity of $\psi(t)$.
\end{proof}

\section{Invariance of \texorpdfstring{$\rho$}{rho} under the non linear flow}

\subsection{Gaussian properties of \texorpdfstring{$\rho$}{rho}}

\begin{proposition} Let $\xi \in L^2$ and $p\geq 1$. There exist two constants $C,a > 0$ such that for all $k$ and all $\Lambda$, we have
$$
\mu_k \Big( \lbrace u \; \Big| \; \|\xi^{1/p} L(\tau) u\|_{L^{2r}_\tau,L^{2p}_x} \geq \Lambda \rbrace \Big) \leq C e^{-a \Lambda}\; ,
$$
and 
$$
\mu \Big( \lbrace u \; \Big| \; \|\xi^{1/p} L(\tau) u\|_{L^{2r}_\tau,L^{2p}_x} \geq \Lambda \rbrace \Big) \leq C e^{-a \Lambda}\; ,
$$
where $L(\tau)$ is the flow of the linear equation $i\partial_t - \sqrt{1-\lap} = 0$, $L^{2r}_\tau$ is the $L^2$  norm on the compact time interval $[-1,1]$, and $a$ depends on $\xi$ like $1/(c\|\xi\|_{L^2})$.
\end{proposition}

\begin{proof} Let $q\geq 2p,2r$. We have 
$$
\mu_k \Big( \lbrace u \; \Big| \; \|\xi^{1/p} L(\tau) u\|_{L^{2r}_\tau,L^{2p}_x} \geq \Lambda \rbrace \Big) = 
\mu_k \Big( \lbrace u \; \Big| \; \|\xi^{1/p} L(\tau) u\|_{L^{2r}_\tau,L^{2p}_x}^q \geq \Lambda^q \rbrace \Big)
$$
which gives, thanks to Markov inequality,
$$
\mu_k \Big( \lbrace u \; \Big| \; \|\xi^{1/p} L(\tau) u\|_{L^{2r}_\tau,L^{2p}_x} \geq \Lambda \rbrace \Big) \leq  \Lambda^{-q} E_{\mu_k} \Big(
 \|\xi^{1/p} L(\tau) u\|_{L^{2r}_\tau,L^{2p}_x}^q\Big)\; .
$$

Let us compute $E_{\mu_k} \Big(
 \|\xi^{1/p} L(\tau) u\|_{L^{2r}_\tau,L^{2p}_x}^q\Big)$. We have 
$$
E_{\mu_k} \Big(
 \|\xi^{1/p} L(\tau) u\|_{L^{2r}_\tau,L^{2p}_x}^q\Big) =  \|\xi^{1/p} L(\tau) \phi_k \|_{L^q_\Omega, L^{2r}_\tau,L^{2p}_x}^q
$$
and thanks to Minkowski inequality
$$
\|\xi^{1/p} L(\tau) \phi_k \|_{L^q_\Omega, L^{2r}_\tau,L^{2p}_x}^q\leq \|\xi^{1/p} L(\tau) \phi_k \|_{ L^{2r}_\tau,L^{2p}_x,L^q_\Omega,}^q \; .
$$
As $\xi^{1/p}(x) L(\tau) \phi_k(x)$ is a Gaussian (at $\tau$ and $x$ fixed), we have 
$$
\|\xi^{1/p} L(\tau) \phi_k \|_{L^q_\Omega} \leq C \sqrt q\|\xi^{1/p} L(\tau) \phi_k \|_{L^2_\Omega}\; .
$$
Since 
$$
L(\tau) \phi_k = \sum_{l= -N_kR_k}^{N_kR_k-1} e^{i(1+\frac{l^2}{N_k^2})\tau} \Big( 1+ \frac{l^2}{N_k^2}\Big)^{-1/2} e^{ilx/N_k} \delta_{N_k,l} \; ,
$$
and since $\xi$ does not depend on the event $\omega \in \Omega$ we have,
$$
\|\xi^{1/p} L(\tau) \phi_k \|_{L^2_\Omega} \leq |\xi(x)|^{1/p} \Big( \int_{0}^{+\infty} \frac{dy}{1+y^2} \Big)^{1/2}\; .
$$

By taking its $L^{2r}_\tau, L^{2p}_x$, we get
$$
\|\xi^{1/p} L(\tau) \phi_k \|_{L^q_\Omega} \leq C \sqrt q \|\xi\|_{L^2}\; .
$$

Therefore,
$$
\mu_k \Big( \lbrace u \; \Big| \; \|\xi^{1/p} L(\tau) u\|_{L^{2r}_\tau,L^{2p}_x} \geq \Lambda \rbrace \Big) \leq (C\sqrt q \Lambda^{-1} \|\xi\|_{L^2})^q \; .
$$

If $\frac{\Lambda^2}{e^2C^2\|\xi\|_{L^2}^2} \geq 2p,2r$, that is, $\Lambda \geq \Lambda_0 = \sqrt{2\max(p,r)} e C \|\xi\|_{L^2}$, then , we can choose $q = \frac{\Lambda^2}{e^2C^2\|\xi\|_{L^2}^2}$, and we get
$$
\mu_k \Big( \lbrace u \; \Big| \; \|\xi^{1/p} L(\tau) u\|_{L^{2r}_\tau,L^{2p}_x} \geq \Lambda \rbrace \Big) \leq e^{-a \Lambda^2} \; 
$$
with $a = \frac{1}{e^2C^2\|\xi\|_{L^2}^2}$. If $\Lambda \leq \Lambda_0$, we have 
$$
\mu_k \Big( \lbrace u \; \Big| \; \|\xi^{1/p} L(\tau) u\|_{L^{2r}_\tau,L^{2p}_x} \geq \Lambda \rbrace \Big) \leq 1 \leq C e^{-a\Lambda^2}
$$
with $C = e^{a\Lambda_0^2}$, which concludes the proof. \end{proof}

\subsection{Measure of A}

\begin{definition}Let $A$ be the set 
$$
A = \bigcup_{\Lambda \geq 1} A(\Lambda) \; .
$$
\end{definition}

\begin{lemma}\label{lem-gaussb}There exists $C$ such that for all $k$ and all $\Lambda \geq 1$, we have
$$
\rho_k(A_k(\Lambda)) \leq \frac{C}{\Lambda^8}\; .
$$
\end{lemma}

\begin{proof} As  $A_k(\Lambda)^c$ is the union of the $A_{k,n}(\Lambda)^c$, thus
$$
\rho_k(A_k(\Lambda)^c) \leq \sum_n \rho_k (A_{k,n}(\Lambda)^c)\; .
$$
Recall that
$$
A_{k,n}(\Lambda) = \lbrace u \in \Hloc \; \Big| \; \| \chi^{1/3}  L(\tau)  \psi_k(t_n) u \|_{L^6_\tau,L^6_x} \leq \Lambda_{n+1} \rbrace \; .
$$
Hence, its complementary is 
$$
A_{k,n}(\Lambda)^c = \lbrace u \in \Hloc \; \Big| \; \| \chi^{1/3}  L(\tau)  \psi_k(t_n) u \|_{L^6_\tau,L^6_x} > \Lambda_{n+1} \rbrace \; .
$$
As the measure $\rho_k$ is invariant under the flow $\psi_k$, we get
$$
\rho_k(A_{k,n}(\Lambda)^c) = \rho_k \left( \lbrace u \in E_k \; \Big| \; \| \chi^{1/3}  L(\tau)   u \|_{L^{q_n}_\tau,L^{q_n}_x} > \Lambda_{n+1} \rbrace \right)
$$
which ensures the bound
$$
\rho_k(A_{k,n}(\Lambda)^c)\leq C(q_n) e^{-a_n \Lambda_{n+1}^2}\; .
$$
We have that $C = C(q)$ depends continuously on $q$ and $q\in [9,12]$ hence $C_n$ can be bounded independently from $n$. What is more, $a_n = \frac{1}{c\|\chi^{q_n/6}\|_{L^2}} \geq c >0$. Hence, we have 
$$
\rho_k(A_{k,n}(\Lambda)^c)\leq C e^{-c \Lambda_{n+1}^2}\; .
$$
By summing it over $n$, we have 
$$
\rho_k(A_k(\Lambda)^c) \leq C \sum_n e^{-c\Lambda^2 (n+1)^{1/4}} \leq \frac{C}{\Lambda^8}\; .
$$\end{proof}

\begin{proposition} The set $A$ is of full $\rho$ measure. \end{proposition}

\begin{proof} 
Since the flow $\psi_k$ is continuous from $\Hloc$ to $\Hloc$ (Proposition \ref{prop-gcpk}) , the set $A_{k,n}^c$ is open in $\Hloc$. Therefore, the complementary set of $A_k(\Lambda)$, which is the union of the $A_{k,n}(\Lambda)^c$ is open. Hence, we have  
$$
\rho (A_{k}(\Lambda)^c \leq \liminf_{j\rightarrow \infty} \rho_j (A_{k}(\Lambda)^c)\; .
$$
Besides, as $A(\Lambda)$ is the $\limsup$ of the $A_k(\Lambda)$, we have $A(\Lambda)^c = \liminf A_k(\Lambda)^c$. We get, thanks to Fatou's lemma
$$
\rho(A(\Lambda)) \leq \liminf_k \rho(A_k(\Lambda)^c) \leq \liminf_k \liminf_j \rho_j( A_k(\Lambda)^c) \; .
$$
Because of the more general property
$$
\liminf_k \liminf_j r_{j,k} \leq \liminf_k r_{k,k} \; ,
$$
we get
$$
\rho(A(\Lambda)^c) \leq \liminf_{k\rightarrow \infty} \rho_k (A_{k}(\Lambda)^c)\; .
$$

Finally, we get
$$
\rho(A(\Lambda)^c) \leq \frac{C}{\Lambda^8}\; ,
$$
and since $A^c$ is the intersection of the $A(\Lambda)^c$ over $\Lambda \geq 1$, we get
$$
\rho(A^c) = 0
$$
and therefore $\rho(A) = 1$, which concludes the proof. \end{proof}

\begin{corollary} The flow $\psi(t)$ is well-defined on the support of $\rho$. \end{corollary}

\subsection{Invariance of \texorpdfstring{$\rho$}{rho} under \texorpdfstring{$\psi$}{psi}}

\begin{theorem}The measure $\rho$ is invariant under the flow $\psi(t)$. That is, for all measurable set $Y$ of $\Hloc$ and all time $t$, we have 
$$
\rho( \psi(t)^{-1}(Y)  ) = \rho(Y)\; .
$$
\end{theorem}

\begin{proof} Let $K$ be a closed set of $\Hloc$ and $t\in \R$. For all $\varepsilon >0$, we call $K_\varepsilon$ the set
$$
K_{\varepsilon} = \{ u \in \Hloc \; |\; \exists v \in K ; d(u,v) < \varepsilon\} \; .
$$
Since $K_\varepsilon$ is open and $\psi(t)$ is continuous in $\Hloc$, $\psi(t)^{-1} (K_\varepsilon)$ is open. Hence, we have, thanks to Proposition \ref{prop-openrho},
$$
\rho (\psi(t)^{-1}(K)) \leq \rho (\psi(t)^{-1}(K_\varepsilon)) \leq \liminf_{k \rightarrow \infty} \rho_k (\psi(t)^{-1}(K_\varepsilon))\; .
$$
Since $K_\varepsilon \subseteq (K_\varepsilon \cap A(\Lambda) )\cup A(\Lambda)^c$, we have 
$$
\rho (\psi(t)^{-1}(K)) \leq \liminf_{k \rightarrow \infty} \left( \rho_k (\psi(t)^{-1}(K_\varepsilon) \cap A(\Lambda)) + \rho_k(A(\Lambda)^c) \right)\; .
$$
Recall that $A(\Lambda)^c = \liminf A_j(\Lambda)$, hence $ \rho_k(A(\Lambda)^c) \leq \liminf_j \rho_k (A_j(\Lambda)^c)$, we get
$$
\rho (\psi(t)^{-1}(K)) \leq \liminf_k \liminf_j  \left( \rho_k (\psi(t)^{-1}(K_\varepsilon) \cap A(\Lambda)) + \rho_k(A_j(\Lambda)^c) \right)\; .
$$
We use again that $\liminf_k \liminf_j r_{j,k} \leq \liminf_k r_{k,k}$, to get
$$
\rho (\psi(t)^{-1}(K)) \leq \liminf_k  \left( \rho_k (\psi(t)^{-1}(K_\varepsilon) \cap A(\Lambda)) + \rho_k(A_k(\Lambda)^c) \right)\; .
$$
Then, we recall that $\rho_k (A_k (\Lambda)^c) \leq \frac{C}{\Lambda^8}$ with  a constant $C$ independent from $k$ (Lemma \ref{lem-gaussb}), to get
$$
\rho (\psi(t)^{-1}(K)) \leq \liminf_k   \rho_k (\psi(t)^{-1}(K_\varepsilon) \cap A(\Lambda)) + \frac{C}{\Lambda^8}\; .
$$
We use then the uniform convergence of $\psi_k$ towards $\psi$ in $A(\Lambda)$ (Proposition \ref{prop-guc}). There exists $k_0$ such that, for all $k \geq k_0$ and all $u \in A(\Lambda)$ 
$$
d( \psi(t)u, \psi_k(t)) \leq \varepsilon \; .
$$
Hence, for all $u \in \psi(t)^{-1}(K_\varepsilon) \cap A(\Lambda)$ and all $k\geq k_0$, there exists $v \in K$ such that $d(\psi(t) u,v) \leq \varepsilon$ and thus
$$
d(\psi_k(t) u , v) \leq d(\psi_k(t)u, \psi(t) u) + d(\psi(t) u , v) < 2\varepsilon \; .
$$
In other worlds, $\psi(t)^{-1}(K_\varepsilon) \cap A(\Lambda)$ is included in $\psi_k(t)^{-1}(K_{2\varepsilon})$ for all $k\geq k_0$. We have
$$
\rho (\psi(t)^{-1}(K)) \leq \liminf_k   \rho_k (\psi_k(t)^{-1}(K_{2\varepsilon})) + \frac{C}{\Lambda^8}\; .
$$
The measure $\rho_k$ is invariant under the flow $\psi_k$, hence 
$$
\rho (\psi(t)^{-1}(K)) \leq \liminf_k   \rho_k (K_{2\varepsilon}) + \frac{C}{\Lambda^8}\; .
$$
The set $K_{2\varepsilon}$ is included in the closed set 
$$
\overline K_{2\varepsilon} = \{ u \in \Hloc \; |\;  d(u,K)\leq 2\varepsilon \} \; ,
$$
and $\liminf \leq \limsup$, hence, we have,
$$
\rho (\psi(t)^{-1}(K)) \leq \limsup_k   \rho_k (\overline K_{2\varepsilon}) + \frac{C}{\Lambda^8}\; ,
$$
and as $\overline K_{2\varepsilon}$ is closed, using again Proposition \ref{prop-openrho},  $\limsup_k  \rho_k (\overline K_{2\varepsilon}) \leq \rho (\overline K_{2\varepsilon})$. We now apply the dominated convergence theorem and let $\varepsilon$ go to $0$, we get, as $K$ is closed
$$
\rho (\psi(t)^{-1}(K)) \leq \rho(K) +  \frac{C}{\Lambda^8}
$$
and then let $\Lambda $ go to $\infty$ to get
$$
\rho (\psi(t)^{-1}(K)) \leq \rho(K)\; .
$$
The reversibility of the flow, along with its continuity on $\Hloc$, ensures that
$$
\rho(K) \leq \rho(\psi(-t)^{-1} \psi(t)^{-1}(K)) \leq \rho (\psi(t)^{-1} (K))
$$
Hence the measure of closed sets is invariant under the flow $\psi(t)$, which we can extend to all measurable sets, as the closed sets generates the topological $\sigma$-algebra of $\Hloc$.
\end{proof}

\bibliographystyle{amsplain}
\bibliography{bibli} 
\nocite{*}

\end{document}